\documentclass[reqno]{amsart}

\usepackage{amsmath}
\usepackage{amssymb}
\usepackage{amsfonts}
\usepackage{amsthm}

\newenvironment{Proof}[1][Proof]
  {\begin{proof}[#1]}
  {\end{proof}}

\usepackage[utf8]{inputenc}
\usepackage[T1]{fontenc}

\usepackage{caption}
\usepackage{subcaption}
\usepackage{tikz}
\usepackage{tikz-cd}
\usepackage{xcolor}
\usepackage{tabularx}
\definecolor{HalfGray}{gray}{0.55}
\definecolor{OliveGreen}{rgb}{0,.35,0}
\definecolor{webbrown}{rgb}{.6,0,0}
\definecolor{BrightViolet}{rgb}{0.5,0.2,0.8}
\definecolor{Maroon}{cmyk}{0, 0.87, 0.68, 0.32}
\definecolor{RoyalBlue}{cmyk}{1, 0.50, 0, 0.25}
\definecolor{Black}{cmyk}{0, 0, 0, 0}

\usepackage{acronym}
\usepackage{paralist}
\usepackage{xspace}

\usepackage[sort&compress,numbers]{natbib}

\newcommand{\citemor}[2][]{\citeauthor{#2} \cite[#1]{#2}}

\usepackage{hyperref}
\hypersetup{
colorlinks=true,
linktocpage=true,
pdfstartview=FitH,
breaklinks=true,
pdfpagemode=UseNone,
pageanchor=true,
pdfpagemode=UseOutlines,
plainpages=false,
bookmarksnumbered,
bookmarksopen=false,
bookmarksopenlevel=1,
hypertexnames=true,
pdfhighlight=/O,
urlcolor=Maroon, linkcolor=RoyalBlue, citecolor=OliveGreen, 
pdftitle={},
pdfauthor={},
pdfsubject={},
pdfkeywords={},
pdfcreator={pdfLaTeX},
pdfproducer={LaTeX with hyperref}
}

\newcommand{\R}{\mathbb{R}}

\newcommand{\Z}{\mathbb{Z}}

\DeclareMathOperator{\bigoh}{\mathcal O}

\DeclareMathOperator{\dom}{dom}

\DeclareMathOperator{\ex}{\mathbb{E}}

\DeclareMathOperator{\grad}{\nabla}

\DeclareMathOperator{\pr}{pr}

\newcommand{\dd}{\:d}

\newcommand{\eps}{\varepsilon}
\newcommand{\exclude}[1]{\operatorname\setminus\left\{#1\right\}}
\newcommand{\from}{\colon}

\newcommand{\pd}{\partial}
\newcommand{\simplex}{\Delta}
\newcommand{\wilde}{\widetilde}

\usepackage[textwidth=32mm]{todonotes}
\usepackage{soul}
\setstcolor{red}

\theoremstyle{plain}
\newtheorem{theorem}{Theorem}
\newtheorem{corollary}[theorem]{Corollary}
\newtheorem*{corollary*}{Corollary}

\newtheorem{proposition}[theorem]{Proposition}

\theoremstyle{definition}
\newtheorem{definition}[theorem]{Definition}
\newtheorem*{definition*}{Definition}

\theoremstyle{remark}
\newtheorem{remark}{Remark}
\newtheorem*{remark*}{Remark}
\newtheorem{example}[theorem]{Example}

\numberwithin{equation}{section}
\numberwithin{theorem}{section}

\DeclareMathOperator{\reg}{Reg}

\newcommand{\braket}[2]{\left\langle #1 \middle\vert  #2 \right\rangle}

\newcommand{\ceil}[1]{\left\lceil #1 \right\rceil}
\newcommand{\floor}[1]{\left\lfloor #1 \right\rfloor}

\newcommand{\bvec}{e}

\newcommand{\set}{\mathcal{A}}
\newcommand{\act}{\mathcal{C}}

\newcommand{\choice}{Q}
\newcommand{\loss}{\ell}
\newcommand{\pay}{u}
\newcommand{\score}{U}
\newcommand{\temp}{\eta}

\renewcommand{\leq}{\leqslant}
\renewcommand{\geq}{\geqslant}
\renewcommand{\subseteq}{\subset}

\DeclareMathOperator*{\argmax}{arg\,max}
\DeclareMathOperator*{\argmin}{arg\,min}

\begin{document}


 \title
[A Continuous-Time Approach to Online Optimization]
{A Continuous-Time Approach\\to Online Optimization}






\author[J. Kwon]{Joon Kwon}
\address{Institut de Math\'{e}matiques de Jussieu, Universit\'{e} Pierre-et-Marie-Curie (UPMC), Paris, France}
\email{joon.kwon@ens-lyon.org}
\urladdr{\url{http://www.math.jussieu.fr/~kwon/}}

\author[P. Mertikopoulos]{Panayotis Mertikopoulos}
\address{French National Center for Scientific Research (CNRS) and La\-bo\-ra\-to\-ire d'In\-for\-ma\-tique de Gre\-no\-ble, Grenoble, France}
\email{panayotis.mertikopoulos@imag.fr}
\urladdr{\url{http://mescal.imag.fr/membres/panayotis.mertikopoulos/home.html}}

\thanks{The authors are greatly indebted to Vianney Perchet for his invaluable help in improving all aspects of this work, and to Rida Laraki and Sylvain Sorin for their insightful suggestions and careful reading of the manuscript.
The authors would also like to express their gratitude to Pierre Coucheney, Bruno Gaujal and Guillaume Vigeral for many helpful discussions and remarks.}
\thanks{Part of this work was done during the authors' visit at the Hausdorff Research Institute for Mathematics at the University of Bonn in the framework of the Trimester Program ``Stochastic Dynamics in Economics and Finance''.
This work was supported by the European Commission in the framework of the FP7 Network of Excellence in Wireless COMmunications NEWCOM\# (contract no. 318306) and the French National Research Agency (ANR) projects GAGA (grant no. ANR-13-JS01-0004-01) and NETLEARN (grant no. ANR-13-INFR-004).}

\begin{abstract}
We consider a family of learning strategies for online optimization problems that evolve in continuous time and we show that they lead to no regret.
From a more traditional, discrete-time viewpoint, this continuous-time approach allows us to derive the no-regret properties of a large class of discrete-time algorithms including as special cases the exponential weight algorithm, online mirror descent, smooth fictitious play and vanishingly smooth fictitious play.
In so doing, we obtain a unified view of many classical regret bounds, and we show that they can be decomposed into a term
stemming from continuous-time considerations and a term which measures the disparity between discrete and continuous time.
As a result, we obtain a general class of infinite horizon learning strategies that guarantee an $\bigoh(n^{-1/2})$ regret bound without having to resort to a doubling trick.
\end{abstract}

\maketitle


\newacro{iid}[i.i.d.]{independent and identically distributed}
\newacro{CDF}{cumulative distribution function}
\newacro{EW}{exponential weight}
\newacro{FP}{fictitious play}
\newacro{FTL}[FtL]{Follow the Leader}
\newacro{FTRL}[FtRL]{Follow the Regularized Leader}
\newacro{SFP}{smooth fictitious play}
\newacro{VSFP}{vanishingly smooth fictitious play}
\newacro{OGD}{online gradient descent}
\newacro{OMD}{online mirror descent}
\newacro{PSG}{projected subgradient}
\newacro{MD}{mirror descent}
\newacro{LMD}{lazy mirror descent}
\newacro{LPSG}{lazy projected subgradient}
\newacro{MDSA}{mirror descent stochastic approximation}
\newacro{SPSG}{stochastic projected subgradient}

\setcounter{tocdepth}{1}
\tableofcontents


\section{Introduction}
\label{sec:introduction}

Online optimization focuses on decision-making in sequentially changing environments (the weather, the stock market, etc.).
More precisely, at each stage of a repeated decision process, the agent/decision-maker
obtains a payoff (or incurs a loss) based on the environment and his decision, and his long-term objective is to maximize his cumulative payoff via the use of past observations.

The worst-case scenario for the agent \textendash\ and one which has attracted considerable interest in the literature \textendash\ is when he has no Bayesian-like prior belief on the environment.
In this context, the cumulative payoff difference between an oracle-like device (a decision rule which prescribes an action based on knowledge of the future) and a \emph{learning strategy} (a rule which only relies on past observations) can become arbitrarily large, even in very simple problems.
As a result, in the absence of absolute payoff guarantees, the most widely used online optimization criterion is that of \emph{regret minimization}, a notion which was first introduced by \citemor{Han57} and has since given rise to a vigorous literature at the interface of optimization, statistics and theoretical computer science
\textendash\ see e.g. \citemor{CBL06}, \citemor{SS11} for a survey.
Specifically, the \emph{cumulative regret} of a strategy compares the payoff obtained by an agent that follows it to the payoff that he would have obtained by constantly choosing one action;
accordingly, one of the main goals in online optimization is to devise strategies that lead to (vanishingly) small average regret against any fixed action, and irrespective of how the agent's environment evolves over time.

In this paper, we take a continuous-time approach to online optimization and we consider a class of strategies that lead to no regret in continuous time.
From a more traditional, discrete-time viewpoint, the importance of this approach lies in that it provides a unifying view of the regret properties of a broad class of well-known online optimization algorithms.
In particular, the discrete-time version of our family of strategies
is an extension of the general class of \ac{OMD} algorithms  (themselves equivalent to ``Following the Regularized Leader''
(\acs{FTRL}) in the case of linear payoffs; see
e.g. \citemor{SS11}, \citemor{Bub11}, \citemor{Haz12}) with a time-varying parameter.
\acused{FTRL}
As such, our analysis contains as special cases
\begin{inparaenum}
[\itshape a\upshape)]
\item
the \ac{EW} algorithm (\citemor{LW94}, \citemor{Vov90}) and its decreasing parameter variant (\citemor{ACBG02});
\item
\ac{SFP} (\citemor{FL99}, \citemor{BHS06}) and \ac{VSFP} (\citemor{BF12});
and
\item
the method of \ac{OGD} introduced by \citemor{Zin03} (the Euclidean predecessor of \ac{OMD}).
\end{inparaenum}

With regards to the \ac{OMD}/\ac{FTRL} family of algorithms, the vanishing regret bounds that we derive by using a time-varying parameter are not particularly new:
bounds of the same order can be obtained by taking existing guarantees for learning with a finite horizon and then using the so-called ``doubling trick'' (\citemor{CBFH+97}, \citemor{Vov95}).%
\footnote{In a nutshell, the doubling trick amounts to breaking up the learning timeline in blocks of exponentially increasing horizon, and then resetting the algorithm at the start of each block with an optimal parameter for the block's (finite) horizon.}
That said, the introduction of a time-varying parameter has several advantages:
\begin{inparaenum}
[\itshape a\upshape)]
\item
it allows us to integrate \ac{SFP} and \ac{VSFP} into the fold and to derive explicit bounds for their regret;
\item
it provides a unified any-time analysis without needing to reboot the algorithm every so often (to the best of our knowledge, such an analysis only exists for the \ac{EW} algorithm with a time-varying parameter (\citemor{Bub11}, \citemor{ACBG02}));
and
\item
in the case of ordinary convex optimization problems with an open-ended termination criterion (as opposed to a fixed number of steps), a variable parameter leads to more efficient value convergence bounds than a variable step-size.
\end{inparaenum}


Building on an idea that was introduced by \citemor{Sor09} in the study of the \acl{EW} algorithm, the key ingredient of our analysis is the descent from continuous to discrete time.
More precisely, given an online optimization problem in discrete time, we construct a continuous-time interpolation where our continuous-time dynamics lead to no regret;
then, by comparing the agent's payoffs in discrete and continuous time, we are able to deduce a bound for the agent's regret in the original discrete-time framework.

One of the main contributions of this approach is that it leads to a unified derivation of several existing regret bounds with disparate proofs;
secondly, it allows us to decompose many classical bounds into two components, a term coming from continuous-time considerations and a comparison term which measures the disparity between discrete and continuous time (see also \citemor{MP13} for an alternative interpretation of such a decomposition).
Each of these terms can be made arbitrarily small by itself, but their sum is coupled in a nontrivial way that induces a trade-off between continuous- and discrete-time considerations:
in a sense, faster decay rates in continuous time lead to greater discrepancies in the discrete/continuous comparison \textendash\ and hence, to slower regret decay bounds in discrete time.


Finally, we also give a brief account of how the derived regret bounds
are related to classical convergence results for certain convex
optimization and stochastic convex optimization algorithms --- including the \ac{PSG} method, \ac{MD}, and their stochastic variants (\citemor{NY83}, \citemor{NJLS09}),
and we illustrate a (somewhat surprising) performance gap incurred by using an optimization algorithm with a decreasing parameter instead of a decreasing step-size.


\subsection{Paper Outline}
\label{sec:outline}

In Section \ref{sec:model}, we present some basics of online optimization to fix notation and terminology;
then, in Section \ref{sec:strategies}, we define regularizer functions, choice maps and the class of
variable-parameter {\ac{OMD}}/{\ac{FTRL}} strategies that we will focus on.
The core of our paper consists of Sections \ref{sec:continuous} and \ref{sec:discrete}:
we first show that the corresponding class of continuous-time strategies leads to no regret in Section \ref{sec:continuous};
this analysis is then translated to discrete time in Section~\ref{sec:discrete} where we derive the no-regret properties of the class of algorithms under consideration.
Finally, in Section \ref{sec:previous}, we establish several links with existing online learning and convex optimization algorithms, and we show how their properties can be derived as corollaries of our results.

\subsection{Notation and Preliminaries}
\label{sec:notation}
Let $d$ be a positive integer and let $V=\R^{d}$ be equipped with an arbitrary norm $\left\| \cdot\right\|$.
The dual of $V$ will be denoted by $V^{\ast}$ and the induced dual norm on $V^{\ast}$ will be given by the familiar expression:
\begin{equation}
\|y\|_{\ast}
	= \sup_{\|x\| \leq 1}
	\left\vert \braket{y}{x} \right\vert,
\end{equation}
where $\braket{y}{x}$ denotes the canonical pairing between $y\in V^{\ast}$ and $x\in V$.
For a nonempty subset $U\subset V$ will use the notation $\|U\| = \sup_{x\in U}\|x\|$.

In the rest of our paper, $\act$ will denote a nonempty compact convex subset of $V$;
moreover, given a convex function $f\from V \to \R\cup\{+\infty\}$, its \emph{effective domain} will be the convex set $\dom f = \{x\in V: f(x) <\infty\}$.
For convenience, if $f\from \act \to \R$ is convex, we will treat $f$ as a convex function on $V$ by setting $f(x) = +\infty$ for $x\in V\setminus \act$;
conversely, if $f\from V\to\R\cup\{+\infty\}$ has domain $\dom f = \act$, we will also treat $f$ as a real-valued function on $\act$ (in all cases, the ambient space $V$ will be clear from the context).
We will then say that $v\in V^{\ast}$ is \emph{a subgradient of $f$ at $x\in \act$} if $f(x') - f(x) \geq \braket{v}{x'-x}$ for all $x'\in \act$;
likewise, the set $\pd f(x) = \{v \in V^{\ast}: \text{$v$ is a subgradient of $f$ at $x$}\}$ will be called the \emph{subdifferential of $f$ at $x$} and $f$ will be called \emph{subdifferentiable} if $\pd f(x)$ is nonempty for all $x\in \dom f$.

If it exists, the minimum (resp. maximum) of a function $f\from V\to \R\cup\{+\infty\}$ will be denoted by $f_{\min}$ (resp. $f_{\max}$).
Moreover, if $\set = \{a_{1},\dots,a_{d}\}$ is a finite set,
the set $\simplex(\set)$ of probability measures on $\set$ will be identified with the standard $(d-1)$-dimensional simplex $\simplex_{d} = \{x\in \R^{d}_{+}: \sum_{i=1}^{d} x_{i}=1\}$ of
$\R^{d}$;
also, the elements of $\set$ will be identified with the corresponding vertices of $\simplex(\set)$, i.e. the canonical basis vectors $\{\bvec_{i}\}_{i=1}^{d}$ of $\R^{d}$.
Finally, for $x,y\in\R$, we will let $\floor{x} = \max\{k\in\Z:k\leq x\}$ and $\ceil{x} = \min\{k\in\Z: k\geq x\}$,
and we will write $x\vee y = \max\{x,y\}$ and $x\wedge y = \min\{x,y\}$.


\section{The Model}
\label{sec:model}

The heart of the online optimization model that we consider is as follows:
at every discrete time instance $n\geq 1$, an agent (decision-maker) chooses an action from a nonempty convex action set $\act\subset V$ and gains a payoff (or incurs a loss) determined by some time-dependent function.
Information about this function is only revealed to the agent after he picks his action, and the agent's objective is to maximize his long-term payoff in an adaptive manner.

\subsection{The Core Model}
\label{sec:model-core}

Let $\act\subset V$ denote the agent's action space.
Then, at each stage $n\geq 1$, the process of play is as follows:
\begin{enumerate}
[1.]
\item
The agent chooses an action $x_{n}\in\act$.

\item
Nature chooses and reveals the \emph{payoff vector} $\pay_{n}\in V^{\ast}$ of the $n$-th stage and the agent receives a payoff of $\braket{\pay_{n}}{x_{n}}$.%
\footnote{Nature may be adversarial, i.e. $\pay_{n}$ may be chosen as a function of $x_{1},\dots,x_{n}$.}

\item
The agent uses some decision rule to pick a new action $x_{n+1}\in \act$ and the process is repeated ad infinitum.%
\end{enumerate}

More precisely, define a \emph{strategy} to be a sequence of maps $\sigma_{n}\from (V^{\ast})^{n-1}\to\act$, $n\geq1$, such that $\sigma_{n+1}$ determines the player's action at stage $n+1$ in terms of the payoff vectors $\pay_{1},\dotsc,\pay_{n}\in V^{\ast}$ that have been revealed up to stage $n$ (in a slight abuse of notation, $\sigma_{1}$ will be regarded as an element of $\act$).
Then, given a sequence of payoff vectors $u=(\pay_{n})_{n\geq1}$ in $V^{\ast}$, the \emph{sequence of actions generated by $\sigma$} will be
\begin{equation}
\label{eq:strategy}
x_{n+1} \equiv \sigma_{n+1}(\pay_{1},\dotsc,\pay_{n}),
\end{equation}
and the agent's \emph{cumulative regret} with respect to $x\in\act$ is defined as:
\begin{equation}
\label{eq:regret-linear}
\begin{aligned}
\reg_{n}^{\sigma,\pay}(x)
	&= \sum_{k=1}^{n} \braket{\pay_{k}}{x}
	-\sum_{k=1}^{n} \braket{\pay_{k}}{x_{k}}
	\\
	&=\sum_{k=1}^{n} \braket{\pay_{k}}{x}
	-\sum_{k=1}^{n} \braket{\pay_{k}}{\sigma_{k}(\pay_{1},\dotsc,\pay_{k-1})}.
\end{aligned}
\end{equation}
In what follows, we focus on strategies that lead to \emph{no} (or, at worst, \emph{small}) \emph{regret}:

\begin{definition}
\label{def:regret}
A strategy $\sigma$ \emph{leads to $\eps$-regret} ($\eps\geq0$) if, for every sequence of payoff vectors $(\pay_{n})_{n\geq 1}$ in $V^*$ such that $\|\pay_{n}\|_{\ast} \leq 1$:
\begin{equation}
\label{eq:eps.regret-linear}
\limsup_{n\to\infty} \frac{1}{n}\max_{x\in \act}\reg_{n}^{\sigma,\pay}(x) \leq \eps.
\end{equation}
In particular, if \eqref{eq:eps.regret-linear} holds with $\eps=0$, we will say that $\sigma$ \emph{leads to no regret}.
\end{definition}

\begin{remark}
The definition of an $\eps$-regret strategy depends on the dual norm $\|\cdot\|_{\ast}$ of $V^{\ast}$ (and hence, on the original norm $\|\cdot\|$ on $V$);
on the other hand, the definition of ``no regret'' is independent of the norm.
\end{remark}

\begin{remark}
In our framework, we can easily see that a strategy leading to $\eps$-regret against ``any sequence'' is equivalent to leading to $\eps$-regret against ``any strategy of nature''.
However, this may not be true in the randomized setting we present in the following paragraph.
\end{remark}

Despite its simplicity, this online linear optimization model may be used to analyze more general online optimization models.
In what follows, we summarize some examples of this kind.

\subsection{The Case of the Simplex and Mixed Actions}
\label{sec:case-simpl-rand}

Consider a discrete decision process where, at each stage $n\geq1$, the agent chooses an action $a_{n}$ from a finite set of \emph{pure} actions $\set=\left\{1,\dotsc,d \right\}$.
To do so, the agent draws $a_{n}$ according to some probability distribution $x_{n}\in \simplex(\set)$;
then, once $a_{n}$ is drawn, the payoff vector $\pay_{n} \in [-1,1]^{d}$ which prescribes the payoff $\pay_{n,a}$ of each action $a\in\set$ is revealed and the agent receives the payoff $\pay_{n,a_{n}}$ that corresponds to his choice of action. 

In this setting, a strategy is still defined as in the core model of Section \ref{sec:model-core} with the agent's action set replaced by the set of \emph{mixed} \emph{actions} $\simplex(\mathcal A)$.%
\footnote{In a more general setting, the choice at each stage might depend not only on the past payoff vectors, but also on the agent's realized actions $a_{1},\dots,a_{n}$.}
The agent's \emph{realized} regret with respect to a pure action
$a\in\set$ will then be
\begin{equation}
\label{eq:regret-realized}
	\sum_{k=1}^{n} (\pay_{k,a} - \pay_{k,a_k}),
\end{equation}
and we will say that a strategy $\sigma$ leads to \emph{$\eps$-realized-regret} (resp. to \emph{no realized regret} for $\eps=0$) if
\begin{equation}
\limsup_{n\to\infty}
	\frac{1}{n}\max_{a\in \mathcal A} \sum_{k=1}^n(u_{k,a}-u_{k,a_k})
	\leq \eps
	\quad
	\text{(a.s.)},
\end{equation}
for every sequence of payoff vectors $(u_n)_{n\geqslant 1}$ in
$\mathbb{R}^d$ such that $\|\pay_{n}\|_{\infty} \leq 1$.%
\footnote{This condition is also called external $\eps$-consistency (\citemor{FL99}, \citemor{BHS06}).}
On the other hand, the agent's expected payoff at stage $n$ is $\ex[\pay_{n,a_n}] = \braket{\pay_{n}}{x_{n}}$;
thus, if we interpret $\braket{\pay_{n}}{x_{n}}$ as the payoff of the \emph{mixed action} $x_{n}\in\simplex_d$, we will have:
\begin{equation}
\ex\left[\sum_{k=1}^{n} \left(\pay_{k,a} - \pay_{k,a_{k}} \right)\right]
	= \sum_{k=1}^{n} \braket{\pay_{k}}{\bvec_{a} - x_{k}}
	= \reg_{n}^{\sigma,\pay}(\bvec_{a})
\end{equation}
where the basis vector $\bvec_{a} \in \simplex(\mathcal A)$ is identified here with the Dirac point mass $\delta_{a}$ on $a\in\mathcal A$.
By a classical argument based on H\oe ffding's inequality and the Borel\textendash Cantelli lemma, the minimization of \eqref{eq:regret-realized} is then reduced to the core model of Section \ref{sec:model-core}:

\begin{proposition}[\citemor{CBL06}, Corollary 4.3]
\label{prop:realized-expected}
If a strategy $\sigma$ leads to $\eps$-regret with respect to the uniform norm on $V^{\ast}$, it also leads to $\eps$-realized-regret.
\end{proposition}

\subsection{Online Convex Optimization}
\label{sec:model-OCO}

We briefly discuss here a more general online convex optimization model where losses are determined by a sequence of convex functions.
Formally, the only change from Section \ref{sec:model-core} is that at each stage $n\geq1$, the agent incurs a loss $\loss_{n}(x_{n})$ determined by a subdifferentiable convex \emph{loss function} $\loss_{n}\from\act\to\R$.
In this nonlinear setting, the information revealed to the agent after playing includes a (negative) subgradient $\pay_{n}\in -\pd\loss_{n}(x_{n}) \subset V^{\ast}$ of $\loss_{n}$ at $x_{n}$, so the incurred cumulative regret with respect to a fixed action $x\in\act$ is:
\begin{equation}
\label{eq:regret-convex}
\sum_{k=1}^n\ell_k(x_k)-\sum_{k=1}^n\ell_k(x).
\end{equation}
By convexity, $\loss_{k}(x') - \loss_{k}(x) \leq \braket{v}{x'-x}$ for all $v\in\pd\loss_{k}(x')$ and for all $x\in\act$;
in this way, \eqref{eq:regret-convex} readily yields:
\begin{equation}
\label{eq:convex-bound}
\sum_{k=1}^n\ell_k(x_k)-\sum_{k=1}^n\ell_k(x)\leqslant 
	-\sum_{k=1}^{n} \braket{\pay_{k}}{x_{k} - x}
	= \sum_{k=1}^{n} \braket{\pay_{k}}{x} - \sum_{k=1}^{n} \braket{\pay_{k}}{x_{k}} 
\end{equation}
where $\pay_{k}\in-\pd\loss_{k}(x_{k})$.
This last expression can obviously be interpreted as the regret incurred by an agent facing a sequence of payoff vectors $\pay_{n} \in V^{\ast}$ (cf. the core model of Section \ref{sec:model-core}),
so a strategy which guarantees a bound on the right-hand side of \eqref{eq:convex-bound} will guarantee the same for \eqref{eq:regret-convex}.
Consequently, when the loss functions $\ell_n$ are uniformly Lipschitz continuous, results for the core model can be directly translated into this one.


\section{Regularizer Functions, Choice Maps and Learning Strategies}
\label{sec:strategies}
\subsection{Regularizer Functions and Choice Maps}
\label{sec:choice}

We begin with the concept of a \emph{regularizer function}:

\begin{definition}
\label{def:regularizer}
A convex function $h\from V\to\R\cup\{+\infty\}$ will be called a \emph{regularizer function on $\act$} if $\dom h = \act$ and $h\vert_{\act}$ is strictly convex and continuous.
\end{definition}

\begin{remark}
  This definition is intimately related to the notion of a
  Legendre-type function (see e.g. \citemor[Section 26]{Roc70}); however, as was recently noted by
  \citemor{SS07} (and in contrast to the analysis of
  e.g. \citemor{BF12}, \citemor{Bub11} and \citemor{BHS06}), we will not require any
  differentiability or steepness assumptions.
\end{remark}

A key tool in our analysis will be the \emph{convex conjugate} $h^{\ast} \from V^{\ast}\to \R\cup\{+\infty\}$ of $h$ defined as
\begin{equation}
\label{eq:conjugate}
h^{\ast}(y)
	= \sup_{x\in V} \{\braket{y}{x} - h(x)\}.
\end{equation}
Since $h$ is equal to $+\infty$ on $V\exclude\act$ and $h\vert_{\act}$ is continuous and strictly convex, the supremum in {\eqref{eq:conjugate}} will be attained at a \emph{unique} point in $\act$.
This unique maximizer then defines our choice map as follows:

\begin{definition}
\label{def:choice}
The \emph{choice map} associated to a regularizer function $h$ on $\act$ will be the map $\choice_{h}\from V^{\ast}\to \act$ defined as
\begin{equation}
\label{eq:choice}
\choice_{h}(y)
	= \argmax_{x\in\act} \{\braket{y}{x} - h(x)\},
	\quad
	y\in V^{\ast}.
\end{equation}
\end{definition}

\begin{example}
[Entropy and logit choice]
\label{ex:entropy}
In the case of the simplex ($\act = \simplex_{d}$),%
\footnote{In this setting, choice maps are more commonly known as \emph{smooth best reply maps} (\citemor{FL98}, \citemor{HS02}, \citemor{BHS06}, \citemor{BF12}).}
a classical example of a choice map is generated by the entropy function
\begin{equation}
\label{eq:regularizer-entropy}
h(x)
	=
	\begin{cases}
	\sum_{i=1}^dx_i\log x_i
	&\text{if $x\in\simplex_{d}$,}
	\\
	+\infty
	&\text{otherwise.}
	\end{cases}
\end{equation}
A standard calculation then yields the so-called \emph{logit choice map}:
\begin{equation}
\label{eq:choice-Gibbs}
\choice_{h}(y)
	= \frac{1}{\sum_{j=1}^{d} e^{y_{j}}}
	\left(e^{y_{1}},\dotsc,e^{y_{d}}\right).
\end{equation}
This map is used to define the \acl{EW} algorithm (cf. Section \ref{sec:previous}), and its importance stems from the well known fact that it leads to the optimal regret bound for $\act = \simplex_{d}$ (\citemor[Theorems 2.2 and 3.7]{CBL06}).
\end{example}

\begin{example}
[Euclidean projection]
\label{ex:Euclidean}
Another important example arises by taking the squared Euclidean distance as a regularizer function;
more precisely, we define the \emph{Euclidean regularizer} on $\act$ as
\begin{equation}
\label{eq:regularizer-Euclidean}
h(x)
	=
	\begin{cases}
	\frac{1}{2} \left\| x \right\|_2^2
	&\text{if $x\in\act$,}
	\\
	+\infty
	&\text{otherwise}.
	\end{cases}
\end{equation}
The associated choice map $\choice_{h}\from\R^{N}\to\act$ corresponds
to taking the orthogonal projection with respect to $\act$:
\begin{flalign}
\label{eq:proj-Euclidean}
\choice_{h}(y)&=\argmax_{x\in\act}
	\big\{
	\braket{y}{x} - \tfrac{1}{2} \|x\|_{2}^{2}
	\big\}
	\notag\\
	&= \argmin_{x\in\act}
	\big\{
	\tfrac{1}{2}\|x\|_{2}^{2} - \braket{y}{x} + \tfrac{1}{2}\|y\|_{2}^{2}
	\big\}
        = \argmin_{x\in\act} \|y-x\|_{2}^{2}.
\end{flalign}
\end{example}

\begin{example}
[Bregman projections]
\label{ex:Bregman}

The Euclidean example above is a special case of a class of projection mappings known as \emph{Bregman projections} (\citemor{Bre67}).

Let $F:V\longrightarrow \mathbb{R}\cup\{+\infty\}$ be a proper convex 
function, differentiable on its domain. Let us denote $\mathcal
D=\operatorname{dom}F$ and for $x,x'\in
\mathcal D$, the \emph{Bregman divergence}
$D_F:\mathcal D\times \mathcal D\longrightarrow \mathbb{R}$ is defined as
\begin{equation}
\label{eq:Breg-divergence}
D_{F}(x,x')
	= F(x) - F(x') - \braket{\nabla F(x')}{x - x'}.
\end{equation}
Hence, given a compact set $\mathcal C\subset \mathcal D$,
the associated \emph{Bregman projection} of a point $x_0\in
\mathcal D$ onto $\mathcal C$ is given by
\begin{equation}
\label{eq:Breg-projection}
\pr_{F}^{\act}(x_0)
	= \argmin_{x\in\act} D_{F}(x,x_0).
\end{equation}
Now assume that $F^*$ is also differentiable on its domain which we
will denote $\mathcal D^*$. It is easy
to check that for $y\in \mathcal D^*$, $\nabla F^*(y)\in
\mathcal D$ and $\nabla F(\nabla F^*(y))=y$. Then, the
process of mapping $y\in \mathcal D^*$ to $\nabla F^*(y)$ and then
projecting to $\mathcal C$ can be written as a choice map in the sense of \eqref{eq:choice}:
\begin{flalign}
\label{eq:Breg-choice}
\pr_{F}^{\act}{\nabla F^{\ast}(y)}
	&= \argmin_{x\in\act} \{F(x) - F(\nabla F^{\ast}(y)) - \braket{\nabla F(\nabla F^*(y))}{x-\nabla F^{\ast}(y)}\}
	\notag\\
	&= \argmin_{x\in\act} \{F(x) - \braket{y}{x}\}
	= \argmax_{x\in\mathbb{R}^d} \{\braket{y}{x} - h(x)\}
	= \choice_{h}(y),
\end{flalign}
where $h\vert_{\act} = F\vert_{\act}$ and $h(x) = +\infty$ for $x\in\R^{d}\exclude\act$.

\end{example}

\subsection{Strategies Generated by Regularizer Functions}
\label{sec:strategies-1}

The class of strategies that we will consider in the rest of this
paper is a variable-parameter extension of the so-called {\acf{OMD}
  method \textendash\ itself equivalent to the family of algorithms
  known as \acf{FTRL} in the case of linear payoffs
  (see e.g. \citemor{SS11} and \citemor{Haz12}).

In a nutshell, this class of strategies may be described as follows:
the agent aggregates his payoffs over time into a score vector $y\in V^{\ast}$ and then uses a choice map to turn these scores into actions and continue playing.
Formally, if $h$ is a regularizer function on the agent's action space $\act$ and $(\temp_{n})_{n\geqslant 1}$ is a positive nonincreasing sequence, \emph{the strategy $\sigma\equiv\big(\sigma_{n}^{h,\temp_{n}}\big)_{n\geq1}$ generated by $h$ with parameter $\temp_{n}$} is defined as
\begin{equation}
\label{eq:AS}
\sigma_{n+1}(\pay_{1},\dotsc,\pay_{n})
	= \choice_{h}\left(\temp_{n} \sum_{k=1}^{n} \pay_{k}\right),
\end{equation}
with $\sigma_{1} = \choice_{h}(0)$.
The corresponding sequence of play $x_{n+1} = \sigma_{n+1}(\pay_{1},\dotsc,\pay_{n})$ will then be given by the recursion:
\begin{equation*}
\begin{aligned}
\score_{n}
	&= \score_{n-1} + \pay_{n},
	\\
x_{n+1}
	&= \choice_{h}(\temp_{n} \score_{n}).
\end{aligned}
\end{equation*}

In addition to the standard variants of \ac{OMD}/\ac{FTRL}, a list of
examples of strategies and algorithms that can be expressed in this
general form is given in Table \ref{tab:algorithms}.
A more detailed analysis (including the regret properties of each algorithm) will also be provided in Section \ref{sec:previous};
we only mention here
that the variability of $\temp_{n}$ will be key for the no-regret properties of $\sigma$:
when $\temp_{n}$ is constant, the strategy \eqref{eq:AS} does not
guarantee a sublinear regret bound (see e.g. \citemor{SS11} and \citemor{Bub11}).

\subsection{Regularity of the Choice Map and the Role of Strong Convexity}
\label{sec:regularity}
In this section, we derive some regularity properties of the choice map $\choice_{h}$ that will be needed in the analysis of the subsequent sections.
We begin by showing that $\choice_{h}$ is continuous and equal to the gradient of $h^{\ast}$:

\begin{proposition}
\label{prop:choice-gradient}
Let $h$ be a regularizer function on $\act$.
Then $h^{\ast}$ is continuously differentiable on $\act$ and $\nabla h^{\ast}(y) = \choice_{h}(y)$ for all $y\in V^{\ast}$.
\end{proposition}

\begin{Proof}
For $y\in V^{\ast}$, we have
\begin{equation}
x\in\pd h^{\ast}(y)
	\iff y \in \pd h(x)
	\iff x \in \argmax\nolimits_{x'\in\act}
	\left\{
	\braket{y}{x'} - h(x')
	\right\},
\end{equation}
i.e. $\pd h^{\ast}(y) = \argmax_{x'\in\act} \{\braket{y}{x'} - h(x')\}$.
However, since the latter set only consists of $\choice_{h}(y)$, $h^{\ast}$ will be differentiable with $\nabla h^{\ast}(y) = \choice_{h}(y)$ for all $y\in V^{\ast}$.
The continuity of $\nabla h^{\ast}$ then follows from \citemor[Corollary 25.5.1]{Roc70}.
\end{Proof}

In the discrete-time analysis of Section \ref{sec:discrete}, \eqref{eq:AS} will be shown to guarantee a regret bound of a simple form when $\choice_{h}$ is Lipschitz continuous.
This last requirement is equivalent to $h$ being \emph{strongly convex}:

\begin{definition}
\label{def:strong}
Let $f\from \mathbb{R}^d\to\R\cup\{+\infty\}$ be a convex function, let $\|\cdot\|$ be a norm on $\R^{d}$, and let $K>0$.
\begin{enumerate}
[(1)]
\item $f$ is
\emph{$K$-strongly convex w.r.t. $\|\cdot\|$} if, for all $w_{1},w_{2}\in \mathbb{R}^d$ and for all $\lambda\in[0,1]$:
\begin{equation}
\label{eq:strong.convex}
f(\lambda w_{1} + (1-\lambda) w_{2})
	\leq \lambda f(w_{1}) + (1-\lambda) f(w_{2})
	- \tfrac{1}{2} K\,\lambda (1-\lambda) \, \|w_{2} - w_{1}\|^{2}.
\end{equation}

\item $f$ is
\emph{$K$-strongly smooth w.r.t. $\|\cdot\|$} if it is differentiable and, for all $w_{1}, w_{2} \in \mathbb{R}^d$:
\begin{equation}
\label{eq:strong-smooth}
f(w_{2})
	\leq f(w_{1}) + \braket{\nabla f(w_{1})}{w_{2} - w_{1}}
	+ \tfrac{1}{2} K\,\|w_{2} - w_{1}\|^{2}.
\end{equation}
\end{enumerate}
\end{definition}


%
%

Strong convexity of a function was shown in \citemor{KSST12} to be equivalent to strong smoothness of its conjugate.
In turn, this equivalence yields the following characterization of Lipschitz continuity:

\begin{proposition}
\label{prop:strong-conv}
Let $f\from V\to\R\cup\{+\infty\}$ be proper and lower semi-continuous.
Then, for $K>0$, the following are equivalent:
\begin{enumerate}
[\textup(i\textup)]
\addtolength{\itemsep}{2pt}
\item
\label{itm:i}
$f$ is $K$-strongly convex with respect to $\|\cdot\|$.

\item
\label{itm:ii}
$f^{\ast}$ is differentiable and $\nabla f^{\ast}$ is $1/K$-Lipschitz.
\item
\label{itm:iii}
$f^{\ast}$ is $1/K$-strongly smooth with respect to $\|\cdot\|_{\ast}$.
\end{enumerate}
\end{proposition}

Hence, given that regularizer functions are proper and lower semi-continuous by definition,
Proposition \ref{prop:strong-conv} leads to the following characterization:

\begin{corollary}
  \label{cor:strongly}
Let $h$ be a regularizer function $\act$ and $K>0$. The associated choice map
$Q_h$ is $K$-Lipschitz continuous if and only if $h$ is $K$-strongly convex
with respect to $\left\| \,\cdot\, \right\|$.
\end{corollary}

This characterization of the Lipschitz continuity of $\nabla f^{\ast}$ (which will be of particular interest to us) is a classical result in the case of the Euclidean norm \textendash\ see e.g. \citemor[Proposition~12.60]{RW98}.
On the other hand, the implication \eqref{itm:ii}$\implies$\eqref{itm:iii} appears to be new in the case of an arbitrary norm (though the proof technique is fairly standard).

\begin{Proof}[Proof of Proposition \ref{prop:strong-conv}]
We will show that \eqref{itm:i}$\implies$\eqref{itm:ii}$\implies$ \eqref{itm:iii}$\implies$\eqref{itm:i}.
\paragraph{$\eqref{itm:i} \implies \eqref{itm:ii}$}
See e.g. \citemor[Proposition~3.1]{BT03}, \citemor[Lemma~1]{Nes09} or \citemor[Lemma~15]{SS07}.

\vspace{1ex}
\paragraph{$\eqref{itm:ii} \implies \eqref{itm:iii}$}
Fix $y_{1}, y_{2} \in V^{\ast}$, let $z = y_{2} - y_{1}$, and set $\phi(t) = f^{\ast}(y_{1} + tz)$, $t\in[0,1]$.
Identifying $V$ with $V^{\ast\ast}$ and $\|\cdot\|_{\ast\ast}$ with $\|\cdot\|$, we have:
\begin{flalign}
\phi'(t) - \phi'(0)
	&= \braket{\nabla f^{\ast}(y_{1}+tz) - \nabla f^{\ast}(y_{1})}{z}
	\notag\\
	&\leq \|z\|_{\ast} \|\nabla f^{\ast}(y_{1}+tz) - \nabla f^{\ast}(y_{1})\|
	\leq \frac{t}{K} \|z\|_{\ast}^{2},
\end{flalign}
where the first inequality follows from the definition of the dual norm and the second from the assumed Lipschitz continuity of $f^{\ast}$.
By integrating, we then get:
\begin{equation}
\phi(t) - \phi(0)
	\leq \phi'(0) t + \frac{1}{2K} t^{2} \|z\|_{\ast}^{2},
\end{equation}
and hence, for $t=1$:
\begin{equation}
f^{\ast}(y_{2}) - f^{\ast}(y_{1})
	\leq \braket{\nabla f^{\ast}(y_{1})}{y_{2} - y_{1}} + \frac{1}{2K} \|y_{2} - y_{1}\|_{\ast}^{2},
\end{equation}
which shows that $f^{\ast}$ is $1/K$-strongly smooth.

\vspace{1ex}
\paragraph{$\eqref{itm:iii} \implies \eqref{itm:i}$}
Since $f$ is proper and lower semi-continuous, it will also be closed.
Our assertion then follows from e.g. \citemor[Theorem 3]{KSST12}.
\end{Proof}

We close this section by stating the strong convexity properties of the regularizer functions of Examples \ref{ex:entropy} and \ref{ex:Euclidean} (which thus imply the Lipschitz continuity of the corresponding choice maps):

\begin{proposition}
\label{prop:entropy-eucl-strong}
With notation as in Examples \ref{ex:entropy} and \ref{ex:Euclidean}, we have:
\begin{enumerate}
[\textup(i\textup)]
\item
The entropy $h\from\simplex_{d}\to\R$ of \eqref{eq:regularizer-entropy} is $1$-strongly convex w.r.t. $\|\cdot\|_{1}$.
\item
The Euclidean regularizer $h\from\act\to\R$ of \eqref{eq:regularizer-Euclidean} is $1$-strongly convex w.r.t. $\|\cdot\|_{2}$.
\end{enumerate}
\end{proposition}

\begin{Proof}
The strong convexity of the Euclidean regularizer is trivial;
for the strong convexity of the entropy with respect to $\|\cdot\|_{1}$, see e.g. \citemor[Proposition 5.1]{BT03}.
\end{Proof}


\section{The Continuous-Time Analysis}
\label{sec:continuous}

Motivated by a technique introduced by \citemor{Sor09} in the context of the \acf{EW} algorithm, we present in this section a continuous-time version of the class of strategies of Section \ref{sec:model} and we derive a bound for the induced regret in continuous time.
This will then enable us to bound the actual discrete-time regret by comparing the continuous- and discrete-time variants of this and the previous section respectively. 

In continuous time, instead of a sequence of payoff vectors $(\pay_{n})_{n\geq1}$ in $V^{\ast}$, the agent will be facing a measurable and locally integrable stream of payoff vectors $(\pay_{t})_{t\in\R_{+}}$ in $V^{\ast}$.
Hence, extending \eqref{eq:AS} to continuous time, we will consider the process:
\begin{equation}
 \label{eq:AS-cont}
x_{t}^{c}
	=\choice_{h}\left(\temp_{t} \int_{0}^{t} \pay_{s} \dd s\right),
\end{equation}
where $(\temp_{t})_{t\in \mathbb{R}_+}$ is a positive, nonincreasing and piecewise continuous parameter, while $x_{t}^{c}\in\act$ denotes the agent's action at time $t$ given the history of payoff vectors $\pay_{s}$, $0\leq s < t$.%
\footnote{In the rest of the paper, we will consistently use $n$ and $k$ for discrete indices and $s,t,\dots$ for continuous ones.}

 Our main result in this section is the following regret bound for \eqref{eq:AS-cont}:

\begin{theorem}
\label{thm:noreg-cont}
If $h$ is a regularizer function on $\act$ and $(\temp_{t})_{t\in \mathbb{R}_+}$ is a positive, nonincreasing and piecewise continuous parameter, then, for every locally integrable payoff stream $(\pay_{t})_{t\in \mathbb{R}_+}$ in
$V^{\ast}$,
we have:
\begin{equation}
\label{eq:reg-bound-cont}
\max_{x\in\act} \int_{0}^{t} \braket{\pay_{s}}{x} \dd s
	-\int_{0}^{t} \braket{\pay_{s}}{x_{s}^{c}} \dd s
	\leq \frac{h_{\max} - h_{\min}}{\temp_{t}}.
\end{equation}
\end{theorem}




\begin{Proof}
Assume first that $\temp_{t}$ is of class $C^1$ and let $y_{t} = \temp_{t} \int_{0}^{t} \pay_{s} \dd s$.
Then, for all $x\in\act$ and for all $t\geq0$, Fenchel's i\-ne\-qua\-li\-ty gives:
\begin{equation}
\label{eq:cum-pay-bound1}
\int_{0}^{t} \braket{\pay_{s}}{x} \dd s
	= \frac{\braket{y_{t}}{x}}{\temp_{t}}
	\leq \frac{h^{\ast}(y_{t}) + h(x)}{\temp_{t}}
	\leq \frac{h^{\ast}(y_{t})}{\temp_{t}} + \frac{h_{\max}}{\temp_{t}}.
\end{equation}
On the other hand, with $x_{t}^{c} = \choice_h(y_{t})$, we will also have by definition:
\begin{equation}
\frac{h^{\ast}(y_{t})}{\temp_{t}}
	= \frac{\braket{y_{t}}{x_t^c} - h(x_t^c)}{\temp_{t}}
	= \int_{0}^{t} \braket{\pay_{s}}{x_{t}^{c}} \dd s - \frac{h(x_{t}^{c})}{\temp_{t}}.
\end{equation}
Consider the function $\phi\from (x,t) \mapsto \int_{0}^{t} \braket{\pay_{s}}{x} \dd s - h(x)/\temp_{t}$.
For fixed $t\geq 0$, one can check that $x_{t}^{c}$ maximizes $\phi(x,t)$, so we can apply the envelope theorem (see e.g. \citemor[Theorem~M.L.1]{MCWG95}) to differentiate $\phi(x_{t}^{c},t)$ with respect to $t$:
\begin{equation}
\label{eq:dh*/dt}
\frac{d}{dt} \frac{h^{\ast}(y_{t})}{\temp_{t}}
	=\frac{\pd \phi}{\partial t}(x_t^c,t)
	= \braket{\pay_{t}}{x_{t}^{c}} + \frac{\dot \temp_{t}}{\temp_{t}^{2}} h(x_{t}^{c})
	\leq \braket{\pay_{t}}{x_{t}^{c}} + h_{\min} \frac{\dot \temp_{t}}{\temp_{t}^{2}},
\end{equation}
where we used the fact that, by assumption, $\dot{\eta}\leqslant 0$.
Integrating \eqref{eq:dh*/dt} then yields
\begin{equation}
\frac{h^{\ast}(y_{t})}{\temp_{t}}
	\leq \frac{h^{\ast}(y_{0})}{\temp_{0}}
	+ \int_{0}^{t} \braket{\pay_{s}}{x_{s}^{c}} \dd s
	+ h_{\min} \int_{0}^{t} \frac{\dot\temp_{s}}{\temp_{s}^{2}} \dd s
	= \int_{0}^{t} \braket{\pay_{s}}{x_{s}^{c}} \dd s
	- \frac{h_{\min}}{\temp_{t}},
\end{equation}
where we have used the fact that $h^{\ast}(y_{0}) = h^{\ast}(0) = -h_{\min}$ in the second step.
Hence, by combining this last equation with \eqref{eq:cum-pay-bound1}, we finally obtain:
\begin{equation}
\label{eq:cum-pay-bound2}
\int_{0}^{t} \braket{\pay_{s}}{x} \dd s
	\leq \int_{0}^{t} \braket{\pay_{s}}{x_{s}^{c}} \dd s
	- \frac{h_{\min}}{\temp_{t}}
	+ \frac{h_{\max}}{\temp_{t}},
\end{equation}
and our claim follows by taking the maximum of the left-hand side over $x\in\act$.

If $\temp_{t}$ is not smooth, let $\temp_{t}^{m}$, $m=1,2\dotsc$, be a sequence of positive and nonincreasing parameters of class $C^{1}$ that converges pointwise to $\temp_{t}$.
Then, if we let $y_{t}^{m} = \temp_{t}^{m} \int_{0}^{t} \pay_{s} \dd s$ and $x_{t}^{m} = \choice_{h}(y_{t}^{m})$, we will also have $x_{s}^{m}\to x_{s}^{c}$ pointwise for all $s\in[0,t]$ by the continuity of $\choice_{h}$.
By the dominated convergence theorem, this implies that $\int_{0}^{t} \braket{\pay_{s}}{x_{s}^{m}} \dd s\to \int_{0}^{t} \braket{\pay_{s}}{x_{s}^{c}} \dd s$ and our assertion follows by the bound \eqref{eq:cum-pay-bound2} for smoothly varying parameters.
\end{Proof}

\begin{remark}
We should note here that the quantity $\delta_h = h_{\max} - h_{\min}$ in \eqref{eq:reg-bound-cont} can be taken arbitrarily small so there is no ``optimal'' regret bound in continuous time.
That said, we shall see in the following section that smaller values of $\delta_h$ result in greater disparities between continuous and discrete time, thus leading to a trade-off for the regret in discrete time.
\end{remark}



\section{Regret Minimization in Discrete Time}
\label{sec:discrete}

In this section, our aim will be to provide a bound for the regret incurred by the discrete-time strategy \eqref{eq:AS}.
To that end, our approach will be as follows:
first, given a positive nonincreasing parameter $(\temp_{n})_{n\geq1}$ and a sequence of payoff vectors $(\pay_{n})_{n\geq1}$, we construct their continuous-time counterparts by setting
\begin{subequations}
\label{eq:interpolation}
\begin{equation}
\label{eq:pay-cont}
\pay_{t} = \pay_{\ceil{t}}
\end{equation}
and
\begin{equation}
\label{eq:temp-cont}
\temp_{t} = \temp_{\floor{t}\vee1}
\end{equation}
\end{subequations}
for all $t\in\R_{+}$ (i.e. $\temp_{t} = \temp_{\floor{t}}$ if $t\geq1$ and $\temp_{t} = \temp_{1}$ otherwise).
Then, given a regularizer $h\from\act\to\R$, we will compare the cumulative payoffs of the processes $(x_{n})_{n\geq1}$ and $(x_{t}^{c})_{t\in\R_{+}}$ that are generated by \eqref{eq:AS} and \eqref{eq:AS-cont} in discrete and continuous time respectively.
In this way, the derived regret bound will consist of two terms:
one coming from the continuous-time bound (\ref{eq:reg-bound-cont}), and a term coming from the discrete/continuous comparison.
Formally:

\begin{theorem}
\label{thm:noreg-disc-strong}
Let $h$ be a $K$-strongly convex regularizer on $\act$ and let $(\temp_{n})_{n\geq1}$ be a positive nonincreasing parameter.
Then, for every sequence of payoff vectors $(\pay_{n})_{n\geq1}$ in $V^{\ast}$, the sequence of play
\begin{equation}
x_{n+1}
	=\choice_{h} \left(\temp_{n} \sum_{k=1}^{n} \pay_{k} \right)
\end{equation}
generated by the strategy $\sigma=(\sigma_n^{h,\temp_{n}})_{n\geq1}$ of \eqref{eq:AS} guarantees the bound
\begin{equation}
\label{eq:reg-disc-unbounded}
\max_{x\in \act}\reg_{n}^{\sigma,\pay}(x)
	\leq \frac{h_{\max}-h_{\min}}{\temp_{n}}
	+ \frac{1}{2K} \sum_{k=1}^{n} \temp_{k-1} \|\pay_{k}\|_{\ast}^{2},
\end{equation}
where we have set $\temp_{0} = \temp_{1}$.
In particular, if $\|\pay_{n}\|_{\ast} \leq M$ for some $M>0$, then:
\begin{equation}
\label{eq:reg-disc-bounded}
\max_{x\in \act} \reg_{n}^{\sigma,\pay}(x)
	\leq \frac{h_{\max} - h_{\min}}{\temp_{n}}
	+ \frac{M^{2}}{2K} \sum_{k=1}^{n} \temp_{k-1}. 
\end{equation}
\end{theorem}

\begin{Proof}
Define the continuous-time interpolations of $\pay_{n}$ and $\temp_{n}$ as in \eqref{eq:interpolation} and let $y_{t} = \temp_{t} \int_{0}^{t} \pay_{s} \dd s$;
Then, for the continuous-time process $x_{t}^{c} = \choice_{h}\left(y_{t}\right)$ generated by \eqref{eq:AS-cont}, we will have:
\begin{equation}
x_{n}
	=\choice_{h}\left(\temp_{n-1} \sum_{k=1}^{n-1} \pay_{k}\right)
	= x_{n-1}^{c},
\end{equation}
and hence, for $k\geq1$ and $t\in(k-1,k)$, the payoffs corresponding to $x_{t}^{c}$ and $x_{k}$ will differ by at most
\begin{flalign}
\label{eq:comp1}
|\braket{\pay_{t}}{x_{t}^{c}} - \braket{\pay_{k}}{x_{k}}|
	&= |\braket{\pay_{k}}{x_{t}^{c} - x_{k-1}^{c}}|
	\notag\\
	&\leq \|\pay_{k}\|_{\ast} \|\choice_{h}(y_{t}) - \choice_{h}(y_{k-1})\|
	\leq \frac{1}{K} \|\pay_{k}\|_{\ast} \, \|y_{t} - y_{k-1}\|_{\ast},
\end{flalign}
where the last inequality follows from the $1/K$-Lipschitz continuity of $\choice_{h}$ (Corollary~\ref{cor:strongly}).
On the other hand, the definition of $y_{t}$ gives
\begin{equation}
\label{eq:comp2}
\|y_{t} - y_{k-1}\|_{\ast}
	= \left\|\temp_{k-1} \int_{k-1}^{t} \pay_{s} \dd s\right\|_*
	\leq \temp_{k-1} \|\pay_{k}\|_{\ast} (t - k + 1),
\end{equation}
which leads to the estimate:
\begin{flalign}
\label{eq:comp3}
\left\vert \int_{0}^{n} \braket{\pay_{t}}{x_{t}^{c}} - \sum_{k=1}^{n} \braket{\pay_{k}}{x_{k}} \right\vert
	&\leq \sum_{k=1}^{n} \int_{k-1}^{k} \vert \braket{\pay_{t}}{x_{t}^{c}} - \braket{\pay_{k}}{x_{k}}\vert \dd t
	\notag\\
	&\leq \frac{1}{K} \sum_{k=1}^{n} \temp_{k-1} \|\pay_{k}\|_{\ast}^{2} \int_{k-1}^{k} (t - k + 1) \dd t
	\notag\\
	&= \frac{1}{2K} \sum_{k=1}^{n} \temp_{k-1} \|\pay_{k}\|_{\ast}^{2}.
\end{flalign}
In view of this discrete/continuous comparison, we thus obtain:
\begin{flalign}
\label{eq:comp4}
\max_{x\in\act} \sum_{k=1}^{n} \braket{\pay_{k}}{x}
	&= \max_{x\in\act} \int_{0}^{t} \braket{\pay_{t}}{x} \dd t
	\notag\\
	&\leq \int_{0}^{n} \braket{\pay_{t}}{x_{t}^{c}} \dd t
	+ \frac{h_{\max} - h_{\min}}{\temp_{n}}
	\notag\\
	&\leq \sum_{k=1}^{n} \braket{\pay_{k}}{x_{k}}
	+ \frac{1}{2K} \sum_{k=1}^{n} \temp_{k-1} \|\pay_{k}\|_{\ast}^{2}
	+ \frac{h_{\max} - h_{\min}}{\temp_{n}},
\end{flalign}
where the first inequality follows from Theorem \ref{thm:noreg-cont} and the last one from \eqref{eq:comp3};
the bounds \eqref{eq:reg-disc-unbounded} and \eqref{eq:reg-disc-bounded} are then immediate.
\end{Proof}

To get the optimal dependence of the bound \eqref{eq:reg-disc-bounded} in $n$, both terms should scale as $\sqrt{n}$ (otherwise, one would be slower than the other).
In this case, we get a bound for the average regret which vanishes as $\bigoh(n^{-1/2})$:

\begin{corollary}
\label{cor:noreg-disc-root}
Let $(\pay_{n})_{n\geq1}$ be a sequence of payoff vectors in $V^{\ast}$.
Then, with notation as in Theorem \ref{thm:noreg-disc-strong},
the sequence of play
\begin{equation}
x_{n+1}
	= \choice_{h}\left(\sqrt{\frac{K (h_{\max} - h_{\min})}{M^{2} n}} \sum_{k=1}^{n} \pay_{k}\right)
\end{equation}
guarantees the regret bound:
\begin{equation}
\label{eq:reg-disc-root}
\max_{x\in\act} \reg_{n}^{\sigma,\pay}(x)
	\leq 2 M\sqrt{\frac{h_{\max} - h_{\min}}{K}}\left( \frac{1}{4} + \sqrt{n} \right).
\end{equation}
\end{corollary}

\begin{Proof}
Set $\delta_h = h_{\max} - h_{\min}$ and $\temp_{n} = \temp/\sqrt{n}$ with $\temp = M^{-1}\sqrt{K \delta_h}$.
Then:
\begin{equation}
\sum_{k=1}^{n} \temp_{k-1}
	= \temp + \temp \sum_{k=1}^{n-1} \frac{1}{\sqrt{k}}
	\leq \temp + \temp \int_{0}^{n-1} \frac{1}{\sqrt{t}} \dd t
	\leq \temp \left(1 + 2\sqrt{n}\right),
\end{equation}
so the bound \eqref{eq:reg-disc-bounded} becomes:
\begin{equation*}
\frac{\delta_h}{\temp_{n}} + \frac{M^{2}}{2K} \sum_{k=1}^{n} \temp_{k-1}
	\leq \frac{\delta_h}{\temp} \sqrt{n} + \frac{M^{2}\temp}{2K} \left(1 + 2 \sqrt{n}\right)
	= 2M\sqrt{\frac{\delta_h}{K}}\left(\frac{1}{4} + \sqrt{n}\right).
\end{equation*}
\end{Proof}

\begin{remark}
We should stress here that regret guarantees of the same order as \eqref{eq:reg-disc-root} can be obtained for the \ac{OMD}/\ac{FTRL} family of algorithms by optimizing the choice of parameter over a finite learning horizon and then restarting the algorithm every so often, using the doubling trick (\citemor{CBFH+97}, \citemor{Vov95}) to guarantee a sublinear regret bound in the long run.
The doubling trick may thus be seen as a special case of a nonincreasing parameter;
for the general case, the bounds \eqref{eq:reg-disc-unbounded}/\eqref{eq:reg-disc-bounded} describe in a precise way the impact of the variability of $\temp_{n}$ on the method's regret guarantees (see also Section \ref{sec:previous} for a more detailed discussion).
\end{remark}

\begin{remark}
The dependence of $\temp$ on $\delta_h$, $K$ and $M$ in \eqref{eq:reg-disc-root} has been chosen precisely so as to minimize the expression $\left(\delta_h/\temp + M^{2}\temp/K\right)$ over all $\temp>0$.
\end{remark}

\begin{remark}
[On the dependence on $K$ and the choice of optimal $h$]
\label{rem:choice-h}
The dependence of the bound \eqref{eq:reg-disc-root} on $K$ is clearly artificial:
\eqref{eq:reg-disc-root} remains invariant if $h$ is rescaled by a positive constant, so it suffices to consider regularizer functions that are $1$-strongly convex over $\act$.
This then leads to the following question:
\emph{given a norm $\|\cdot\|$ on $V$ and a compact convex subset $\act\subset V$, which $1$-strongly convex function minimizes $h_{\max} - h_{\min}$?}
With the exception of the Euclidean norm, this question does not seem to admit a trivial answer (cf. Section~\ref{sec:optimal-choice-h} for a more detailed discussion).
\end{remark}

By expressing the cumulative payoff gap between discrete- and continuous-time \emph{exactly}, Theorem \ref{thm:noreg-disc-strong} can be extended further to regularizer functions that are not strongly convex over $\act$.
The only thing that changes in this case is that the comparison term of the bound \eqref{eq:reg-disc-bounded} is replaced by a term involving the Bregman divergence associated with the convex conjugate
$h^{\ast}$ of $h$.

The following result is a variable-parameter extension of Theorem 5.6 in \citemor{bubeck2012regret}.
\begin{theorem}
\label{thm:noreg-disc-precise}
Let $h$ be a regularizer function on $\act$.
Then, with notation as in Theorem \ref{thm:noreg-disc-strong}, the strategy $\sigma = (\sigma^{h,\temp_{n}}_n)_{n\geq1}$ of \eqref{eq:AS} guarantees the regret bound:
\begin{equation}
\label{eq:reg-disc-precise}
\max_{x\in\act} \reg_{n}^{\sigma,\pay}(x)
	\leq \frac{h_{\max} - h_{\min}}{\temp_{n}}
	+ \sum_{k=1}^{n} \frac{1}{\temp_{k-1}} D_{h^{\ast}} (y_{k}^{-},y_{k-1}^{+}),
\end{equation}
where we have set $y_{n}^{+} = \temp_{n} \sum_{k=1}^{n} \pay_{k}$, $y_{n}^{-} = \temp_{n-1} \sum_{k=1}^{n} \pay_{k}$ and $\temp_{0} = \temp_{1}$.
\end{theorem}

\begin{Proof}
With notation as in the proof of Theorem \ref{thm:noreg-disc-strong}, the variables $y_{n}^{\pm}$ in the statement of the theorem may be expressed more concisely as:
\begin{equation}
y_{n}^{\pm}
	= \lim_{t\to n^{\pm}} y_{t}
	= \lim_{t\to n^{\pm}} \temp_{t} \int_{0}^{t} \pay_{s} \dd s,
\end{equation}
and hence, with $\temp_{t}$ right-continuous, we get $x_{n} = \choice_{h}(y_{n-1}) = \choice_{h}(y_{n-1}^{+})$.
Accordingly, if $x_{t}^{c} = \choice_{h}(y_{t})$ denotes the continuous-time process generated by \eqref{eq:AS-cont}, then, for all $k\geq1$ and for all $t\in(k-1,k)$, we will have:
\begin{equation}
\braket{\pay_{t}}{x_{t}^{c}} - \braket{\pay_{k}}{x_{k}}
	= \braket{\pay_{t}}{\choice_{h}(y_{t})} - \braket{\pay_{k}}{\choice_{h}(y_{k-1}^{+})}
	= \braket{\pay_{k}}{\grad h^{\ast}(y_{t})} - \braket{\pay_{k}}{\grad h^{\ast}(y_{k-1}^{+})}.
\end{equation}
In this way, noting that $\braket{\pay_{t}}{\grad h^{\ast}(y_{t})}$ is simply the derivative of $h^{\ast}(y_{t})/\temp_{k-1}$ for $t\in(k-1,k)$, we obtain the following comparison over $(k-1,k)$:
\begin{flalign}
\int_{k-1}^{k} \braket{\pay_{t}}{x_{t}^{c}} \dd t - \braket{\pay_{k}}{x_{k}}
	&= \int_{k-1}^{k} \frac{1}{\temp_{k-1}} \frac{d}{dt} \left(h^{\ast}(y_{t}) \right) \dd t
	- \frac{1}{\temp_{k-1}} \braket{\temp_{k-1} \pay_{k}}{\grad h^{\ast}(y_{k-1}^{+})}
	\notag\\
	&= \frac{1}{\temp_{k-1}} \left( h^{\ast}(y_{k}^{-}) - h^{\ast}(y_{k-1}^{+}) - \braket{y_{k}^{-} - y_{k-1}^{+}}{\grad h^{\ast}(y_{k-1}^{+})}\right)
	\notag\\
	& = \frac{1}{\temp_{k-1}} D_{h^{\ast}}\left( y_{k}^{-},y_{k-1}^{+} \right).
\end{flalign}
In view of the above, the claim follows by summing this bound over $k=1,\dotsc,n$ and plugging the resulting expression in the first inequality of \eqref{eq:comp4} \textendash\ which holds independently of any assumptions on $h$.
%
\end{Proof}




\section{Links with Existing Results}
\label{sec:previous}

In this section, we discuss how certain existing results in online optimization and (stochastic) convex programming can be obtained as corollaries of the general analysis of the previous sections.

\subsection{Links with Known Online Optimization Algorithms}

\subsubsection{The Exponential Weight Algorithm}
\label{sec:EW}

The \acf{EW} algorithm was introduced independently by \citemor{LW94} and \citemor{Vov90} as a learning strategy in discrete time.
Motivated by the approach of \citemor{Sor09} who used a continuous-time variant to retrieve the algorithm's classical regret bounds, we show here how the same bounds can be obtained directly from Theorem \ref{thm:noreg-disc-strong}.

The framework of the \ac{EW} algorithm is that of randomized action selection as in Section \ref{sec:case-simpl-rand}.
Specifically, let $\set = \{1,\dotsc,d\}$ be a finite set of \emph{pure} actions, and let the agent's action set be the unit simplex $\act =\simplex_{d}$ of $\R^{d}$ \textendash\ the latter being endowed with the $\ell^{1}$ norm $\|\cdot\|_{1}$.
In this context, the \ac{EW} algorithm is defined as:
\begin{equation}
\label{eq:EW}
\tag{\acs{EW}}
\begin{aligned}
\score_{n}
	&= \score_{n-1} + \pay_{n},
	\\
x_{i,n+1}
	&= \frac{e^{\temp\score_{i,n}}}{\sum_{j=1}^{d} e^{\temp \score_{j,n}}}
\end{aligned}
\end{equation}
where $\temp>0$ is a (fixed) parameter and $(\pay_{n})_{n\geq1}$ is a sequence of payoff vectors in $[-1,1]^{d}$ (so that $\|\pay_{n}\|_{\infty} \leq 1$ in the induced dual norm).

Example \ref{ex:entropy} in Section \ref{sec:choice} shows that \eqref{eq:EW} corresponds to \eqref{eq:AS} with $\temp_{n} = \temp$ and $h(x) = \sum_{i=1}^{d} x_{i} \log x_{i}$.
Since $h_{\max} - h_{\min} = \log d$ and $h$ is $1$-strongly convex
with respect to $\|\cdot\|_{1}$ (cf. Proposition~\ref{prop:entropy-eucl-strong}), Theorem~\ref{thm:noreg-disc-strong} readily yields the bound
\begin{equation}
\label{eq:reg-bound-EW}
\max_{a\in\set} \reg_{n}(a)
	\leq \frac{\log d}{\temp} + \frac{n\temp}{2}.
\end{equation}
Additionally, if the time horizon $n$ is known in advance, the optimal
parameter choice $\temp =\sqrt{2 \log d/n}$ leads to
\begin{equation}
\label{eq:reg-bound-EW-finite}
\max_{\alpha\in\set} \reg_{n}(a)
	\leq \sqrt{2n\log d},
\end{equation}
which, as far as the dependence on $d$ and $n$ is concerned, is the best possible bound a strategy can guarantee in this framework \textendash\ see e.g. \citemor[Theorem~3.7]{CBL06}.

\begin{remark}
  By taking $\pay_{n}\in[0,1]^d$ (as is often the case in the
  literature) and then shifting to $[-1/2,1/2]^{d}$,
  Theorem~\ref{thm:noreg-disc-strong} can be applied with $M=1/2$.
  This yields a factor of $1/8$ in the second term of
  \eqref{eq:reg-bound-EW} and leads to the bound obtained by
  \citemor{CB97} and \citemor{CBL06}.
\end{remark}

\subsubsection{The Exponential Weight Algorithm with $\temp_{n} = 1/\sqrt{n}$}
\label{sec:EW'}

\citemor{ACBG02} considered the following variant of \eqref{eq:EW}
\begin{equation}
\label{eq:EW'}
\tag{\acs{EW}$'$}
\begin{aligned}
\score_{n}
	&= \score_{n-1} + \pay_{n},
	\\
x_{i,n+1}
	&= \frac{e^{\temp \score_{i,n}/\sqrt{n}}}{\sum_{j=1}^{d} e^{\temp \score_{j,n}/\sqrt{n}}}.
\end{aligned}
\end{equation}
In our context, a direct application of Corollary \ref{cor:noreg-disc-root} with $M=K=1$ then gives
\begin{equation}
\label{eq:reg-bound-EW'}
\max_{a\in\set} \reg_{n}(a)
	\leq 2 \sqrt{ n \log d}+\frac{1}{2}\sqrt{\log d},
\end{equation}
a bound which, unlike \eqref{eq:reg-bound-EW-finite}, has the advantage of holding uniformly in time.

\subsubsection{Smooth Fictitious Play}
\label{sec:SFP}

The \acf{SFP} process was introduced by \citemor{FL95} (see also
\citemor{FL98} and \citemor{FL99}), and its regret properties were examined further by \citemor{BHS06} using the theory of stochastic approximation \textendash\ but without providing any quantitative bounds for the regret.

Just like the \ac{EW} algorithm, \ac{SFP} falls within the randomized actions framework of Section \ref{sec:case-simpl-rand}.
In particular, \ac{SFP} corresponds to the sequence of play generated by \eqref{eq:AS} for an arbitrary regularizer on $\simplex_{d}$ and with parameter $\temp/n$ for some $\temp>0$ ;
specifically:
\begin{equation}
\label{eq:SFP}
\tag{\acs{SFP}}
x_{n+1} = \choice_{h}\left(\frac{\temp}{n} \sum_{k=1}^{n} \pay_{k}\right).
\end{equation}
With regards to the regret induced by \eqref{eq:SFP}, \citemor[Theorem 6.6]{BHS06} show that for every $\eps>0$, there exists some $\temp^{\ast} \equiv \temp^{\ast}(\eps)$ such that the strategy \eqref{eq:SFP} with parameter $\temp\geq\temp^{\ast}$ leads to $\eps$-realized-regret.
On the other hand, combining Proposition~\ref{prop:realized-expected} with Theorem~\ref{thm:noreg-disc-strong} yields the following more precise statement:

\begin{proposition}
\label{prop:SFP}
Let $h$ be a $K$-strongly convex regularizer on the unit simplex $\simplex_{d}\subset\R^{d}$ endowed with the $\ell^{1}$ norm.
Then, for every sequence of payoff vectors $(\pay_{n})_{n\geq1}$ in $[-1,1]^{d}$, the strategy \eqref{eq:SFP} with parameter $\temp>0$ guarantees
\begin{equation}
\label{eq:reg-bound-SFP}
\max_{a\in\set} \reg_{n}(a)
	\leq \frac{h_{\max} - h_{\min}}{\temp} n
	+ \frac{\temp \log n}{2K}
	+\frac{\temp}{K}.
\end{equation}
In particular, \eqref{eq:SFP} with parameter $\temp$ leads to $(h_{\max} - h_{\min})/\temp$ \textup(realized\textup) regret.
\end{proposition}

\begin{Proof}
Simply combine the logarithmic growth estimate $\sum_{k=1}^{n} k^{-1} < 1+\log n$ for the harmonic series and Theorem \ref{thm:noreg-disc-strong} with $\temp_{n} = \temp/n$;
the claim for the realized regret then follows from Proposition \ref{prop:realized-expected}.
\end{Proof}

\begin{remark}
  It should be noted here that the qualitative analysis of
  \citemor{BHS06} does not require $h$ to be strongly convex; that said,
  if $h$ is strongly convex, Proposition \ref{prop:SFP} gives a
  quantitative bound on the regret.
\end{remark}

\subsubsection{Vanishingly Smooth Fictitious Play}
\label{sec:VSFP}

The variant of \ac{SFP} known as \acf{VSFP} was introduced by \citemor{BF12}, and its regret properties were established using sophisticated tools from the theory of differential inclusions and stochastic approximation \textendash\ but, again, without providing explicit regret bounds.

Using the same notation as before, \ac{VSFP} corresponds to the sequence of play
\begin{equation}
\label{eq:VSFP}
\tag{\acs{VSFP}}
 x_{n+1}
 	=\choice_{h}\left(\temp_{n}\sum_{k=1}^{n} \pay_{k}\right),
\end{equation}
where $h$ is a strongly convex regularizer on $\simplex_{d}$ and the sequence $\temp_{n}$ satisfies:

\begin{enumerate}[(\textup{A}1)]
\addtolength{\itemsep}{.5ex}
\item
\label{itm:A1}
$\lim_{n\to\infty} n \temp_{n} = +\infty$.
\item
\label{itm:A2}
$\temp_{n} = \bigoh(n^{-\alpha})$ for some $\alpha>0$.
\end{enumerate}
Under these assumptions, the main result of \citemor{BF12} is that \eqref{eq:VSFP} leads to no realized regret;
in our framework, this follows directly from Proposition \ref{prop:realized-expected} and Theorem~\ref{thm:noreg-disc-strong} (which also gives a quantitative regret guarantee):

\begin{proposition}
\label{prop:VSFP}
With notation as in Proposition \ref{prop:SFP},
the strategy \eqref{eq:VSFP} with $\temp_{n}$ satisfying assumptions \textup{(A1)} and \textup{(A2)} guarantees the regret bound
\begin{equation}
\label{eq:reg-bound-VSFP-general}
\max_{a\in\set} \frac{1}{n} \reg_{n}(a)
	\leq \frac{h_{\max} - h_{\min}}{n\temp_{n}}
	+ \frac{1}{2nK} \sum_{k=1}^{n} \temp_{k-1},
\end{equation}
and thus leads to no regret.
In particular, if $\temp_{n} = \temp n^{-\alpha}$ for some $\alpha\in(0,1)$, then:
\begin{equation}
\label{eq:reg-bound-VSFP}
\max_{a\in\set} \frac{1}{n} \reg_{n}(a)
	\leq \frac{h_{\max}-h_{\min}}{\temp n^{1-\alpha}}
        + \frac{\temp n^{-\alpha}}{2 (1-\alpha) K}
        + \frac{\temp}{2Kn}.
\end{equation}
\end{proposition}

\begin{Proof}
The bound \eqref{eq:reg-bound-VSFP-general} is an immediate corollary of Theorem \ref{thm:noreg-disc-strong};
the no-regret property then follows from Assumptions $(\textup{A}1)$ and $(\textup{A}2)$.
Finally, if $\temp_{n} = \temp n^{-\alpha}$, we get
\begin{equation}
 \sum_{k=1}^n\eta_{k-1}
 	= 1 + \sum_{k=1}^{n-1} k^{-\alpha}
	\leq 1 + \int_{0}^{n-1} t^{-\alpha} \dd t
	= 1 + \frac{n^{1-\alpha}}{1-\alpha},
\end{equation}
and \eqref{eq:reg-bound-VSFP} follows by substituting the above in \eqref{eq:reg-bound-VSFP-general}.
\end{Proof}

\begin{remark}
  If we take $h(x) = \sum_{i=1}^dx_i\log x_i$ and $\alpha = 1/2$,
  \eqref{eq:VSFP} boils down to \eqref{eq:EW'}; the bound
  \eqref{eq:reg-bound-EW'} then also follows from
  \eqref{eq:reg-bound-VSFP}.
\end{remark}

\subsubsection{Online Gradient Descent}
\label{sec:OGD}

The \acf{OGD} algorithm was introduced by \citemor{Zin03} in the context of online convex optimization that we described in Section \ref{sec:model-OCO} \textendash\ see also \citemor[Section 4.1]{Bub11}.
Here, we focus on a so-called \emph{lazy} variant (\citemor[p.~144]{SS11}) defined by means of the recursion
\begin{equation}
\label{eq:OGD}
\tag{OGD-L}
\begin{aligned}
\score_{n}
	&\in \score_{n-1} - \temp\,\pd \loss_{n}(x_{n}),
	\\
x_{n+1}
	&= \argmin_{x\in \act} \|x - \score_{n}\|^{2},
\end{aligned}
\end{equation}
where $\loss_{n}\from\act\to\R$ is a sequence of $M$-Lipschitz loss functions, $\temp>0$ is a constant parameter, and the algorithm is initialized with $\score_{0} = 0$.

In view of Example \ref{ex:Euclidean}, \eqref{eq:OGD} corresponds to the strategy $\sigma=(\sigma_n^{h,\eta})_{n\geqslant 1}$ generated by the Euclidean regularizer $h$ on $\act$ \textendash\ defined itself as in \eqref{eq:regularizer-Euclidean}.
Theorem \ref{thm:noreg-disc-strong} thus yields the regret bound:
\begin{equation}
\label{eq:reg-bound-OGD}
\max_{x\in\act} \frac{1}{n} \reg_{n}(x)
	\leq \frac{\delta_{\act}^{2}}{2n\temp}
	+ \frac{\temp M^{2}}{2}
\end{equation}
with $\delta_{\act}^{2} = \max_{x\in\act} \|x\|_{2}^{2} - \min_{x\in\act} \|x\|_{2}^{2}$.
Accordingly, if the time horizon $n$ is known in advance, the optimal choice for $\temp$ is $\temp = \delta_{\act} /( M \sqrt{n})$, leading to a cumulative regret guarantee of $M \delta_{\act} \sqrt{n}$, which is essentially the bound derived by \citemor[Corollary.~2.7]{SS11} (see also \citemor[Theorem~3.1]{Bub11} for the greedy variant).%
\footnote{For the difference between lazy and greedy variants, see Section {\ref{sec:greedy-versus-lazy}}.}

\subsubsection{Online Mirror Descent}
\label{sec:OMD}

The family of (lazy) \acf{OMD} algorithms studied by Shalev-Shwartz \cite{SS07,SS11} is the most general family of strategies that we discuss in this section (see also \citemor{Bub11} for a greedy version).
In particular, the \ac{OMD} class of strategies contains \ac{EW} and \ac{OGD} as special cases, and it is also equivalent to the family of \acf{FTRL} algorithms in the case of linear payoffs (\citemor{SS11}, \citemor{Haz12}).

Following \citemor{SS11} (and with notation as in Section \ref{sec:model-OCO}), let $\ell_{n}\from\act\to\R$ be a sequence of convex functions which are $M$-Lipschitz with respect to some norm $\|\cdot\|$ on $\R^{d}$.
Then, given a regularizer function $h$ on $\act$, the lazy \ac{OMD} algorithm is defined by means of the recursion:
\begin{equation}
\label{eq:OMD}
\tag{OMD-L}
\begin{aligned}
\score_{n}
	&\in \score_{n-1} - \temp\,\pd\loss_{n}(x_{n}),
	\\
x_{n+1}
	&= \choice_{h}(U_{n}),
\end{aligned}
\end{equation}
where $\temp>0$ is a \emph{fixed} parameter and the algorithm is initialized with $\score_{0} = 0$.
As a result, if $h$ is taken $K$-strongly convex with respect to $\|\cdot\|$, Theorem~\ref{thm:noreg-disc-strong} immediately yields the known regret bound for \ac{OMD}:
\begin{equation}
\max_{x\in\act} \reg_{n}(x)
	\leq \frac{h_{\max} - h_{\min}}{\temp}
	+ \frac{\temp M^{2} n}{2K}.
\end{equation}

\subsection{Links with Convex Optimization}
\label{sec:convex}
Ordinary convex programs can be seen as online optimization problems where the loss function remains constant over time and the agent seeks to attain its minimum value.
In what follows, we outline how regret-minimizing strategies can be used for this purpose
and we describe the performance gap incurred by using a method with a variable step-size instead of a variable parameter.

Let $f\from\act\to\R$ be a convex real-valued function on $\act$ and let $(\gamma_{n})_{n\geq1}$ be a positive sequence (which we will later interpret as a sequence of step-sizes);
also, given a sequence $(x_{n})_{n\geq1}$ in $\act$, let
\begin{align}
\label{eq:x-mod}
x_{n}^{\min}
	&\in \argmin_{1\leq k \leq n} f(x_{k}),
	&
x_{n}^{\gamma}
	&= \frac{\sum_{k=1}^{n} \gamma_{k} x_{k}}{\sum_{k=1}^{n}
  \gamma_{k}}.
\end{align}
If we use the notation $x_{n}'\in\{x_{n}^{\min},x_{n}^{\gamma}\}$ to refer interchangeably to either $x_{n}^{\min}$ or $x_{n}^{\gamma}$, Jensen's inequality readily gives:
\begin{equation}
\label{eq:x-mod-comp}
f(x_{n}')
	\leq \frac{\sum_{k=1}^{n} \gamma_{k} f(x_{k})}{\sum_{k=1}^{n}\gamma_{k}}.
\end{equation}
Now consider the algorithm:
\begin{equation}
\label{eq:algo-cvx}
\begin{aligned}
\score_n
	&\in \score_{n-1} - \gamma_{n} \pd f(x_{n}),\\
x_{n+1}
	&= \choice_{h}(\temp_{n} \score_{n}),
\end{aligned}
\end{equation}
where $\gamma_{n}$ is a sequence of step sizes and $\temp_{n}$ is a sequence of parameters.
In the case of a constant parameter $\eta_n=1$, \eqref{eq:algo-cvx} then becomes
\begin{equation}
\label{eq:LMD}
\tag{MD-L}
\begin{aligned}
\score_{n}
	&\in \score_{n-1} - \gamma_{n} \pd f(x_{n}),
	\\
x_{n+1}
	&= \choice_{h}(\score_{n}).
\end{aligned}
\end{equation}
which is a lazy variant of the \acf{MD} algorithm (\citemor{NY83}). In particular, if $h$ is the Euclidean regularizer on $\act$, the algorithm boils down to a lazy version of the standard \acf{PSG} method:
\begin{equation}
\tag{PSG-L}
\label{eq:LPSG}
\begin{aligned}
\score_{n}
	&\in \score_{n-1} - \gamma_{n} \pd f(x_{n}),
	\\
x_{n+1}
	&= \argmin_{x\in\act} \|x - \score_{n}\|_{2}.
\end{aligned}
\end{equation}

The following corollary shows that these lazy versions guarantee the same value convergence bounds as the corresponding greedy variants \textemdash\ see e.g.~\citemor[Theorem~4.1]{BT03}.

\begin{corollary}
[Constant parameter, variable step size]
\label{cor:varstep}
Let $f\from\act\to\R$ be an $M$-Lipschitz convex function and let $(x_{n})_{n\geq1}$ be the sequence of play generated by \eqref{eq:LMD} for some $K$-strongly convex regularizer $h$ on $\act$.
Then, the adjusted iterates $x_{n}' \in \{x_{n}^{\min},x_{n}^{\gamma}\}$ of $x_{n}$ satisfy:
\begin{equation}
\label{eq:conv-bound-varstep}
f(x_{n}')
	\leq f_{\min}
	+ \frac{h_{\max} - h_{\min} + \frac{1}{2}M^{2} K^{-1}\sum_{k=1}^{n} \gamma_{k}^{2}}{\sum_{k=1}^{n} \gamma_{k}}.
\end{equation}
\end{corollary}
\begin{Proof}
With $\sigma=(\sigma_n^{h,\eta_n})_{n\geqslant 1}$, $\pay_{k} \in - \gamma_{k} \pd f(x_{k})$ and $x_{n}'\in\{x_{n}^{\min}, x_{n}^{\gamma}\}$, we have:
\begin{flalign}
\label{eq:value-bound1}
\reg_{n}^{\sigma,\pay}(x)	
	= \sum_{k=1}^{n} \braket{\pay_{k}}{x - x_{k}}
	\geq - \sum_{k=1}^{n} \gamma_{k} \big(f(x) - f(x_{k})\big)
	\geq \sum_{k=1}^{n} \gamma_{k}\cdot \big(f(x_{n}') - f(x)\big),
\end{flalign}
where the last step follows from \eqref{eq:x-mod-comp}.
By taking $x\in\argmin f$, we then obtain:
\begin{equation}
\label{eq:value-bound2}
f(x_{n}') - f_{\min}
	\leq \frac{\reg_{n}^{\sigma,\pay}(x)}{\sum_{k=1}^{n} \gamma_{k}}.
\end{equation}
The result then follows by applying Theorem~\ref{thm:noreg-disc-strong} and using the fact that $\|\pay_{k} \|_{\ast} \leq \| \gamma_{k} \partial f(x_{k}) \|_{\ast} \leq \gamma_{k} M$ (recall that $f$ is $M$-Lipschitz continuous).
\end{Proof}

One can see that the best convergence rate that we get with constant $\temp$
and step-sizes of the form $\gamma_{n} \propto n^{-\alpha}$ is $\bigoh(\log n/\sqrt{n})$ for $\alpha=1/2$ (and there is no straighforward choice of $\gamma_{n}$ leading to a better convergence rate).
On the other hand, by taking a \emph{constant} step-size $\gamma_{n} = 1$ and \emph{varying the algorithm's parameter} $\temp_{n} \propto n^{-1/2}$, we do achieve an $\bigoh(n^{-1/2})$ rate of convergence.

\begin{corollary}
[Constant step size, variable parameter]
\label{cor:vartemp}
With notation as in Corollary~\ref{cor:varstep}, let $(x_{n})_{n\geq1}$ be the sequence of play generated by \eqref{eq:algo-cvx} with
\begin{equation}
\label{eq:temp-opt}
\temp_{n}
	= \frac{1}{M}\sqrt{\frac{K(h_{\max} - h_{\min})}{n}},
\end{equation}
and constant $\gamma_{n} = 1$.
Then, the adjusted iterates $x_{n}' \in \{x_{n}^{\min},x_{n}^{\gamma}\}$ of $x_{n}$ guarantee
\begin{equation}
\label{eq:value-bound-vartemp}
f(x_{n}')
	\leq f_{\min} + 2M \sqrt{\frac{h_{\max} - h_{\min}}{K}}\left( \frac{1}{\sqrt{n}}+\frac{1}{4n} \right).
\end{equation}
\end{corollary}

\begin{Proof}
Similar to the proof of Corollary \ref{cor:varstep}.
\end{Proof}

\subsection{Noisy Observations and Links with Stochastic Convex Optimization}
 \label{sec:stochastic}

Assume that at every stage $n=1,2,\dotsc$ of the decision process, the agent does not observe the actual payoff vector $\pay_{n} \in V^{\ast}$, but the realization of a random vector $\wilde\pay_{n}$ with $\ex\big[ \wilde\pay_{n}|\wilde\pay_{n-1},\dotsc, \wilde\pay_{1} \big] = \pay_{n}$.
In this case, a learning strategy $\sigma$ can be used with the observed vectors $\wilde\pay_{n}$, thus leading to a (random) sequence of play $\wilde x_{n+1} = \sigma_{n+1}(\wilde\pay_{1},\dotsc,\wilde\pay_{n})$
\textendash\ see e.g. \citemor[Section 4.1]{SS11} for a model of this kind.

In this framework, the agent's (maximal) cumulative regret will be given by
\begin{equation}
\max_{x\in\act} \sum_{k=1}^{n} \braket{\pay_{k}}{x - \wilde x_{k}}.
\end{equation}
On the other hand, $\sum_{k=1}^{n} \braket{\wilde\pay_{k}}{x - \wilde x_{k}}$ can be interpreted as the agent's cumulative regret against the observed payoff sequence $(\wilde\pay_{n})_{n\geq1}$.
Thus, if $h$ is a $K$-strongly convex regularizer on $\act$ and $\|\wilde\pay_{k}\|_{\ast} \leq M$ (a.s.), Theorem \ref{thm:noreg-disc-strong} yields:
\begin{equation}
\ex\left[\max_{x\in\act} \sum_{k=1}^{n} \braket{\pay_{k}}{x - \wilde x_{k}}\right]
	= \max_{x\in\act} \sum_{k=1}^{n} \left[\ex\braket{\wilde\pay_{k}}{x - \wilde x_{k}} - \ex\braket{\wilde\pay_{k} - \pay_{k}}{x - \wilde x_{k}}\right]
	\leq R_{n},
\end{equation}
where $R_{n}$ is the regret guarantee of \eqref{eq:reg-disc-bounded} and we have used the easily verifiable fact that $\ex[\braket{\wilde\pay_{k} - \pay_{k}}{\wilde x_{k} - x}] = 0$ (recall that $\ex[\wilde\pay_{k} \vert \wilde\pay_{k-1},\dotsc,\wilde\pay_{1}] = 0$ and that $\wilde x_{k}$ only depends on $\wilde\pay_{k-1},\dotsc,\wilde\pay_{1}$).
This bound is of the same form as that of e.g. \citemor[Theorem 4.1]{SS11};
furthermore, by the strong law of large numbers for martingale differences (\citemor[Theorem~2.18]{HH80}), we also obtain the stronger statement that
\begin{equation}
\max_{x\in\act} \sum_{k=1}^{n} \braket{\pay_{k}}{x - \wilde x_{k}}
	\leq R_{n} + o(n)
	\quad\text{(a.s.)},
\end{equation}
i.e., if the parameter $\temp_{n}$ is suitably chosen, then the strategy \eqref{eq:AS} with noisy observations still leads to no regret (a.s.).

The above can be adapted to the framework of stochastic convex optimization as follows:
let $f\from\act\to\R$ be a Lipschitz convex function on $\act$, let $(\gamma_{n})_{n\geq1}$ be a positive sequence of step sizes, and consider the strategy $\sigma$ generated by \eqref{eq:AS} with $\temp=1$ and $h$ a $K$-strongly convex regularizer on $\act$.
Then, the sequence of play
\begin{equation}
\label{eq:AS-stoch}
\wilde x_{n+1} = \sigma_{n+1}(-\gamma_{1}\wilde g_{1},\dotsc,-\gamma_{n} \wilde g_{n})
	= \choice_{h}\left(-\sum_{k=1}^{n}\gamma_{k}\wilde g_{k} \right)
\end{equation}
where $\wilde g_{n}$ is a random vector with $\ex\big[ \wilde g_{n} | \wilde g_{n-1},\dotsc,\wilde g_{1}\big] \in \partial f(\wilde{x}_n)$ may be written recursively as:
\begin{equation}
\label{eq:MDSA-L}
\tag{MDSA-L}
\begin{aligned}
\wilde \score_{n}
	&\in \wilde \score_{n-1} - \gamma_{n}\pd f(\wilde x_{n}),
	\\
\wilde x_{n+1}
	&= \choice_{h}(\wilde \score_{n}).
\end{aligned}
\end{equation}

This algorithm may be seen as a lazy version of the so-called \ac{MDSA} process of \citemor{NJLS09};
in particular, using the Euclidean regularizer leads to the lazy \ac{SPSG} method:
\begin{equation}
\label{eq:SPSG-L}
\tag{SPSG-L}
\begin{aligned}
\wilde \score_{n}
	&\in \wilde \score_{n-1} - \gamma_{n}\pd f(\wilde x_{n}),
	\\
\wilde x_{n+1}
	&= \argmin_{x\in \act}\| x-\wilde{U}_n \|_2 .
\end{aligned}
\end{equation}
Setting $\pay_{n} = -\gamma_{n} g_{n}$, $\wilde\pay_{n} = -\gamma_{n} \wilde g_{n}$ and taking $\wilde x_{n}'\in\{\wilde x_{n}^{\min},\wilde x_{n}^{\gamma}\}$ as before, Corollary~\ref{cor:varstep} combined with our previous discussion then gives:
\begin{flalign}
\ex\left[f(\wilde x_{n}') - f(x)\right]
	&\leq \frac{1}{\sum_{k=1}^{n} \gamma_{k}} \sum_{k=1}^{n} \ex\braket{\pay_{k}}{x - \wilde x_{k}}
	\notag\\
	&\leq \frac{h_{\max} - h_{\min} + \frac{1}{2} M^{2} K^{-1} \sum_{k=1}^{n} \gamma_{k}^{2}}{\sum_{k=1}^{n} \gamma_{k}},
\end{flalign}
which is essentially the same value guarantee as that of greedy
\ac{MDSA} (\citemor[Eq.~2.41]{NJLS09}).

\begin{table}[t]
\begin{center}
\sc
\renewcommand{\arraystretch}{1.6}
\begin{tabular}{cccccc}
\hline
algorithm
	&\hspace{.5em} $\act$ \hspace{.5em}
	&\hspace{1em} $h(x)$ \hspace{1em}
	&$\temp_{n}$
	&input
	&\sc norm
	\\
\hline
\hline
\ref{eq:EW}
	&$\Delta_d$
	&$\sum_{i} x_{i}\log x_{i}$
	&constant
	&$\pay_{n}$
	&$\ell^{1}$
	\\
\hline
\ref{eq:EW'}
	&$\Delta_{d}$
	&$\sum_{i} x_{i}\log x_{i}$
	&$\temp/\sqrt{n}$
	&$\pay_{n}$
	&$\ell^{1}$
	\\
\hline
\ref{eq:SFP}
	&$\Delta_d$
	&any
	&$\temp/n$
	&$\pay_{n}$
	&$\ell^{1}$
	\\
\hline
\ref{eq:VSFP}
	&$\Delta_d$
	&any
	&$\temp n^{-\alpha}
	\;\;(0<\alpha<1)$
  	&$\pay_{n}$
	&$\ell^{1}$
	\\
\hline
\ref{eq:OGD}
	&any
	&$\frac{1}{2} \left\| x \right\|_2^2 $
	&constant
	&$-\nabla f_{n}(x_{n})$
	&$\ell^{2}$
	\\
\hline
\ref{eq:OMD}
	&any
	&any
	&constant
	&$-\nabla f_{n}(x_{n})$
	&any
	\\
\hline
\ref{eq:LPSG}
	&any
	&$\frac{1}{2} \left\| x \right\|_2^2 $
	&$1$
	&$-\gamma_{n}\nabla f(x_{n})$
	&$\ell^{2}$
	\\
\hline
\ref{eq:LMD}
	&any
	&any
	&$1$
	&$-\gamma_{n} \nabla f(x_{n})$
	&any
	\\
\hline
\ref{eq:MDSA-L}
	&any
	&any
	&$1$
	&$-\gamma_{n} (\nabla f(x_{n})+\xi_n)$
	&any
	\\
\hline
\ref{eq:SPSG-L}
	&any
	&$\frac{1}{2} \left\| x \right\|_2^2 $
	&$1$
	&$-\gamma_{n}(\nabla f(x_{n})+\xi_n)$
	&$\ell^{2}$
	\\
\hline
\end{tabular}
\vspace{2ex}
\end{center}
\caption{%
\footnotesize
Summary of the algorithms discussed in Section~\ref{sec:previous}.
The suffix ``L'' indicates a ``lazy'' variant;
the \textsc{input} column stands for the stream of payoff vectors which is used as input for the algorithm and the \textsc{norm} column specifies the norm of the ambient space;
finally, $\xi_{n}$ represents a zero-mean stochastic process with values in $\R^{d}$.}
\label{tab:algorithms}
\end{table}


\section{Discussion}
\label{sec:discussion}

\subsection{On the optimal choice of $h$}
\label{sec:optimal-choice-h}
As mentioned in the discussion after Corollary \ref{cor:noreg-disc-root}, the following open question arises:
\emph{given a norm $\|\cdot\|$ on $V$ and a compact, convex subset
  $\act\subseteq V$, which $1$-strongly convex regularizer on
  $h\from\act\to\R$ has minimal depth $\delta_h = h_{\max} - h_{\min}$?}

As the following proposition shows, in the case of the Euclidean norm
on $V$, this minimal depth is half the radius squared of the smallest enclosing sphere of $\act$:


\begin{proposition}
Let $h\from\act\to\R$ be a $1$-strongly convex regularizer function on $\act$ with respect to the $\ell^{2}$ norm $\|\cdot\|_{2}$ on $V$.
Then:
\begin{equation}
\label{eq:depth-opt-Euclidean}
h_{\max} - h_{\min}
	\geq \frac{1}{2} \min_{x'\in\act} \max_{x\in\act} \|x' - x\|_{2}^{2},
\end{equation}
and equality is attained by taking
\begin{equation}
\label{eq:h-opt-Euclidean}
h(x)
	=\begin{cases}
	\frac{1}{2} \|x - x_{0}\|_{2}^{2}
		&\text{if $x\in\act$},
	\\
	+\infty
		&\text{otherwise},
	\end{cases}
\end{equation}
where $x_{0} \in \argmin_{x'\in\act} \max_{x\in\act} \|x' - x\|_{2}^{2}$ is the center of the smallest enclosing sphere of $\act$.
\end{proposition}

\begin{proof}
Letting $x_{1} \in \argmin_{x\in\act} h(x)$ and $x_{2} \in \argmax_{x\in\act} \|x - x_{1}\|_{2}^{2}$,
we readily get:
\begin{flalign}
h_{\max} - h_{\min}
	&\geq h(x_{2}) - h(x_{1})
	\notag\\
	&\geq \frac{1}{2} \|x_{2} - x_{1}\|_{2}^{2}
	= \frac{1}{2} \max_{x\in\act} \|x - x_{1}\|_{2}^{2}
	\geq \frac{1}{2} \min_{x'\in\act} \max_{x\in\act} \|x - x'\|_{2}^{2},
\end{flalign}
where the second inequality follows from the strong convexity of $h$ and the fact that $\pd h(x_{1}) \ni 0$.
That \eqref{eq:h-opt-Euclidean} attains the bound \eqref{eq:depth-opt-Euclidean} is then a trivial consequence of its definition, as is its geometric characterization.
%
%
\end{proof}

Despite the simplicity of the bound \eqref{eq:depth-opt-Euclidean}, this analysis does not work for an arbitrary norm because $\frac{1}{2}\left\| x-x_0 \right\|^2$ might fail to be $1$-strongly convex with respect to $\|\cdot\|$
\textendash\ for instance, $\|x - x_{0}\|_{1}^{2}$ is not even \emph{strictly} convex.

\subsection{Greedy versus Lazy}
\label{sec:greedy-versus-lazy}


%


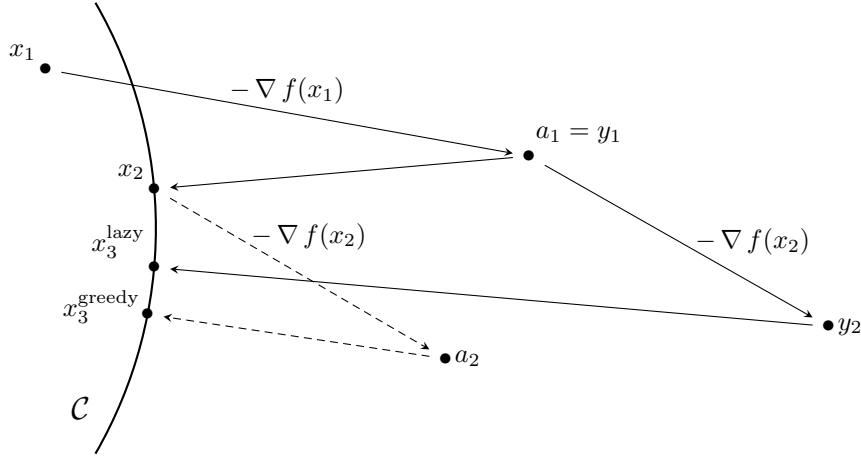
\begin{figure}[t]
  \centering
\begin{tikzpicture}
[scale=2]

\draw node at (2.5,-1.2) {\Large$\act$};
\draw [thick] (30:3) arc (30:-30:3);

\node (xn) at (25:2.5) {\textbullet};
\node [above left] at (25:2.5) {$x_{1}$};
\node (an1) at (5:5.5) {\textbullet};
\node [right] at (6.5:5.5) {$a_{1} = y_{1}$};
\draw[-stealth] (xn) -- (an1) node [midway,above]{$-\grad f(x_{1})$};

\node (xn1) at (5:3) {\textbullet};
\node [above left] at (5:3) {$x_{2}$};
\draw [-stealth] (an1) -- (xn1);

\node (an2) at (-10:5) {\textbullet};
\node (an22) at (-5:7.5) {\textbullet};
\node[right] at (-5:7.5) {$y_{2}$};
\node [right] at (-10:5) {$a_{2}$};

\draw [densely dashed, -stealth] (xn1) -- (an2) node [near start, right] {\;$-\grad f(x_{2})$};
\node (xn2) at (-11:3) {\textbullet};

\draw [-stealth] (an1) -- (an22) node[midway, right]{\;$-\grad f(x_{2})$};
\node [left] at (-10:3) {$x_{3}^{\text{greedy}}$};
\draw [densely dashed, -stealth] (an2) -- (xn2);

\node (xn22) at (-5:3) {\textbullet};
\node[above left] at (-5:3) {$x_{3}^{\text{lazy}}$};
\draw[-stealth] (an22) -- (xn22);
\end{tikzpicture}

\caption{Graphical illustration of the greedy (dashed) and lazy (solid) branches of the \acf{PSG} method.}
\label{fig:comparison}
\end{figure}

To illustrate the difference between \emph{lazy} and \emph{greedy} variants, we first focus on the \ac{PSG} method run with constant step $\gamma=1$ for a smooth function $f\from\act\to\R$.
The two variants may then be expressed by means of the recursions:
\begin{subequations}
\begin{equation}
\begin{aligned}
a_{n}
	&= x_{n} - \grad f(x_{n})
	\\
x_{n+1}
	&= \argmin_{x\in\act} \|x - a_{n}\|_{2}
\end{aligned}
\end{equation}
for the greedy version and:
\begin{equation}
\begin{aligned}
y_{n}
	&= y_{n-1} - \grad f(x_{n})
	\\
x_{n+1}
	&= \argmin_{x\in\act} \|x - y_{n}\|_{2}
\end{aligned}
\end{equation}
\end{subequations}
for the lazy one.

As can be seen in Fig.~\ref{fig:comparison}, the greedy variant is based on the classical idea of gradient descent, i.e. adding $-\grad f(x_{n})$ to $x_{n}$ and projecting back to $\act$ if needed.
On the other hand, in the lazy variant, the gradient term $-\grad f(x_{n})$ is \emph{not} added to $x_{n}$, but to the ``unprojected'' iterate $y_{n}$;
we only project to $\act$ in order to obtain the algorithm's next iterate.
Owing to this modification, the lazy variant is thus driven by the sum $y_{n} = \sum_{k=1}^{n} \grad f(x_{n})$.

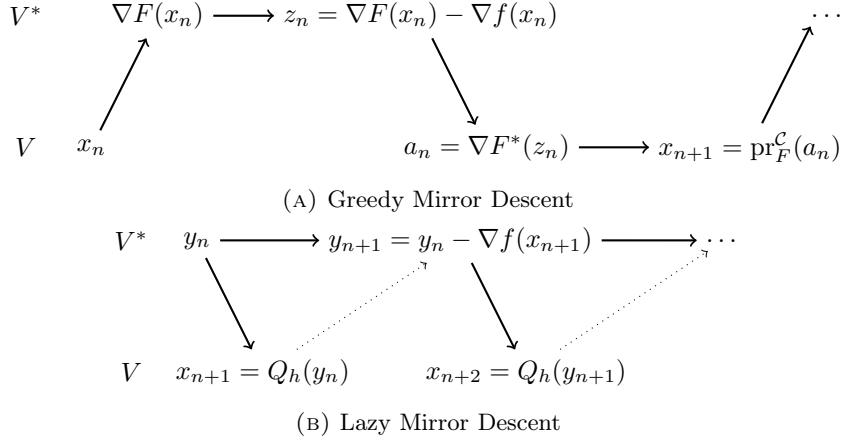
\begin{figure}[t]
  \centering
  \begin{subfigure}{\textwidth}
    \centering
    \begin{tikzpicture}[scale=1.75]
      \node (v) at (-.5,0) {$V$}; \node (vstar) at (-.5,1) {$V^*$};
      \node (xn) at (0,0) {$x_n$}; \node (nfxn) at (0.5,1) {$\nabla
        F(x_n)$}; \node (zn) at (2.5,1) {$z_n=\nabla F(x_n)-\nabla
        f(x_n)$}; \node (an) at (3,0) {$a_{n}=\nabla F^*(z_n)$}; \node
      (xn1) at (5,0) {$x_{n+1}=\pr_F^\act(a_{n})$}; \node (etc) at
      (5.5,1) {$\quad\dots$}; \draw [->,thick] (xn) -- (nfxn)
      ; \draw [->,thick] (nfxn) -- (zn)
      ; \draw [->,thick] (zn) -- (an)
      ; \draw [->,thick] (an) -- (xn1)
      ; \draw [->,thick] (xn1) -- (etc)
      ;
    \end{tikzpicture}
    \caption{Greedy Mirror Descent}
    \label{fig:gmd}
  \end{subfigure}
  \begin{subfigure}{\textwidth}
    \centering
    \begin{tikzpicture}[scale=1.75]
      \node (v) at (-.5,0) {$V$}; \node (vstar) at (-.5,1) {$V^*$};
      \node (yn) at (0,1) {$y_n$}; \node (xn1) at (.5,0)
      {$x_{n+1}=Q_h(y_n)$}; \node (yn1) at (2,1) {$y_{n+1}=y_n-\nabla
        f(x_{n+1})$}; \node (xn2) at (2.5,0) {$x_{n+2}=Q_h(y_{n+1})$};
      \node (etc) at (4,1) {$\dots$}; \draw [->,thick] (yn) -- (xn1) ;
 \draw [->,thick] (yn) -- (yn1) ;
 \draw [->,thick] (yn1) -- (xn2) ;
 \draw [->,thick] (yn1) -- (etc) ;
 \draw [->,dotted] (xn1) -- (yn1);
 \draw [->,dotted] (xn2) -- (etc);
    \end{tikzpicture}
    \caption{Lazy Mirror Descent}
    \label{fig:lmd}
  \end{subfigure}
\caption{Greedy and Lazy Mirror Descent with $\gamma_n=1$.}
\label{fig:greedy-lazy-md}
\end{figure}

In the case of \acl{MD} with an arbitrary regularizer function $h$, the lazy version has an implementation advantage over its greedy counterpart.
Specifically, given a proper convex function $F$ such that $F = h$ on $\act$ (cf. Example~\ref{ex:Bregman}), greedy mirror descent is defined as:
\begin{subequations}
\begin{equation}
\begin{aligned}
a_{n}
	&= \grad F^{\ast}\left(\grad F(x_{n}) - \grad f(x_{n})\right),
	\\
x_{n+1}
	&= \pr_{F}^{\act} (a_{n}),
\end{aligned}
\end{equation}
where the Bregman projection $\pr_F^{\act}(a_n)$ is given by \eqref{eq:Breg-projection};
on the other hand, lazy \ac{MD} is defined as
\begin{equation}
\begin{aligned}
y_{n}
	&= y_{n-1} - \grad f(x_{n}),
	\\
x_{n+1}
	&= \choice_{h}(y_{n}).
\end{aligned}
\end{equation}
\end{subequations}
The computation steps for each variant are represented in Figure~\ref{fig:greedy-lazy-md}.
The first step in the greedy version which consists in computing $\nabla F$ has no equivalent in the lazy version, which is thus computationally more lightweight.




\footnotesize
\bibliographystyle{ormsv080}
\bibliography{OnlineOptimization}

\end{document}